\documentclass[a4paper,12pt,oneside,reqno]{amsart}
\usepackage{hyperref}
\usepackage[headinclude,DIV13]{typearea}
\areaset{15.1cm}{25.0cm}
\usepackage{txfonts,amssymb,amsmath,amsthm,bbm,color}
\usepackage[T1]{fontenc}
\usepackage{amsmath}
\usepackage{amsfonts}
\usepackage{amssymb}
\usepackage{graphicx}
\usepackage{float}

\usepackage{csquotes}
\usepackage{amsmath,amssymb,amstext,amsthm}
\usepackage{bbm}
\usepackage{hyperref}
\usepackage{lmodern}
\usepackage{here}
\numberwithin{equation}{section}

\usepackage{subcaption}
\theoremstyle{plain}
\newtheorem{theorem}{Theorem}[section]
\newtheorem{lem}[theorem]{Lemma}
\newtheorem{prop}[theorem]{Proposition}
\newtheorem{cor}[theorem]{Corollary}
\newtheorem{rem}[theorem]{Remark}
\newtheorem{mydef}[theorem]{Definition}

\begin{document}
	
	%%%%%%%%% TITLE PAGE %%%%%%%%%%%%%%%%%%%%%%%%%%

	\title[BBM among obstacles]{Maximum of Branching Brownian Motion among mild obstacles}

	\author[L. Hartung]{Lisa Hartung}
	\address{L. Hartung\\
		Institut f\"ur Mathematik \\Johannes Gutenberg-Universit\"at Mainz\\
		Staudingerweg 9,
		55099 Mainz, Germany}
	\email{lhartung@uni-mainz.de}
	\author[M. Lehnen]{Mich\`ele Lehnen}
	\address{M. Lehnen\\
		Institut f\"ur Mathematik \\Johannes Gutenberg-Universit\"at Mainz\\
		Staudingerweg 9,
		55099 Mainz, Germany}
	\email{lhartung@uni-mainz.de}
	
	\date{\today}

	\begin{abstract} 
	We study the height of the maximal particle at time $t$ of a one dimensional branching Brownian motion with a space-dependent branching rate. The branching rate is set to zero in finitely many intervals ({\it obstacles}) of order $t$. We obtain almost sure asymptotics of the first order of the maximum, describe the path of a particle reaching this height and describe its dependence on the size and location of the obstacles. 
		
	\end{abstract}
	
	\thanks{
		This work was partly funded by the Deutsche Forschungsgemeinschaft (DFG, German Research Foundation) under Germany's Excellence Strategy - GZ 2047/1, Projekt-ID 390685813 and GZ 2151 - Project-ID 390873048, through Project-ID 233630050 -TRR 146, through Project-ID 443891315 within SPP 2265, and Project-ID 446173099.
	}

	
	\maketitle

\section{Introduction}\label{section introduction}
Standard branching Brownian motion is a prototype for a spatial branching process, which has been studied extensively in the last decades also due to its connection with the F-KPP equation \cite{fisher37, kpp, Ikeda1, Ikeda2, Ikeda3}. It was shown by Bramson \cite{B_M,B_C} that the position of the maximal particle at time $t$ is tight around 
\begin{equation}\label{max.hom}
m(t)=\sqrt{2} t -\frac{3}{2\sqrt{2}} \log(t)
\end{equation}
and  the convergence of the extremal process was proven in \cite{ABK_E,ABBS}. There are several ways to introduce inhomogeneities  into branching Brownian motion. Branching Brownian motion with  time inhomogeneous variance has been studied extensively in \cite{FZ_RW, BovierHartung2014, BovierHartung2015, FZ_BM, MZ, BH18}. Certain instances of branching Brownian motion with space inhomogeneous branching rate have been analysed in \cite{BBHHR15, RS21}, where the branching rate is a function of the distance to the origin. Moreover, branching Brownian motion with (mild) obstacles has been studied in \cite{Engl2015,OME17,EH03} focusing mainly on the total population size. In the present article, we consider a one dimensional branching Brownian motion for time $t$, in which the branching is suppressed in space intervals of order $t$. 

\subsection{The model}	
In this article, we study a one dimensional branching Brownian motion (BBM) with space inhomogeneous branching rate. More precisely, we consider a BBM for a time horizon $t$, that does not branch in some space intervals of order $t$ and otherwise branches at rate 1 into two.

	\begin{mydef} \label{Def.model}
		Let $\ell\in\mathbb{N}$. For $i=1,\!...,\ell$, let $a_{i}>0$ and $b_{i}>0$ be some constants. We call $(a_{i},b_{i})_{i=1}^{\ell}$ an obstacle landscape (see  Figure \ref{fig:picture-obstacles}). BBM among obstacles is denoted by $\{X_{k}(t),k=1,\!...,n(t)\}$ and defined as a one dimensional dyadic BBM with space dependent branching rate $\mathbbm{1}_{K^{\prime}}$ where $K^{\prime}$ is the complement of
		\begin{equation}
			K=\bigcup_{m=1}^{\ell}\left(\sum_{i=1}^{m-1}(a_{i}+b_{i})t+a_{m}t,\sum_{i=1}^{m}(a_{i}+b_{i})t\right).
		\end{equation}
	\end{mydef}

	\begin{figure}[H]
		\centering
		\includegraphics[width=1\linewidth]{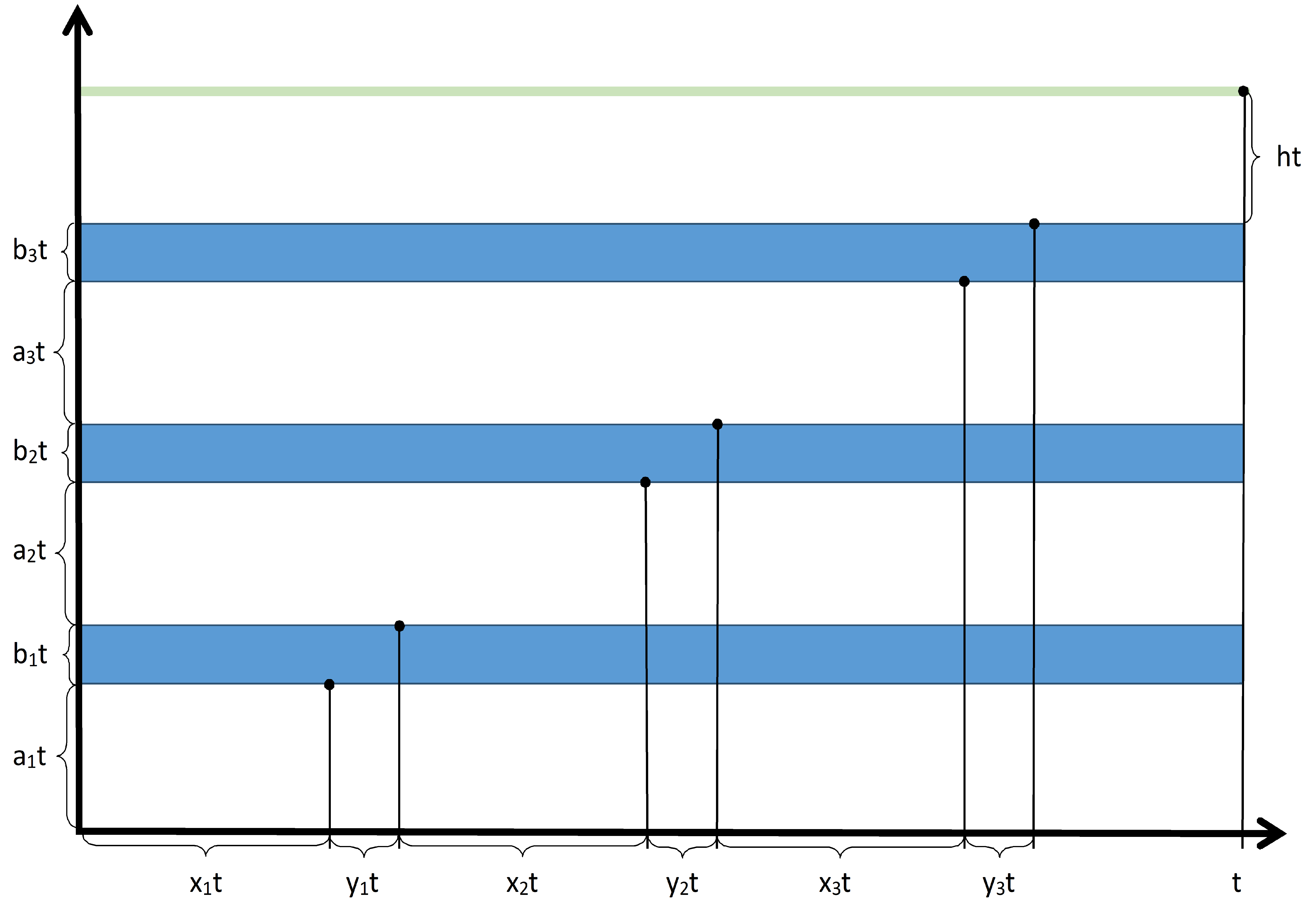}
		\caption{This is an example of $\ell=3$ obstacles. The vertical axis shows space and the horizontal axis time. In the white intervals of size $a_{i}t$ and in $[\sum_{i=1}^{\ell}(a_{i}+b_{i})t,\infty)$ and $(-\infty,0)$, BBM among obstacles branches at rate $1$. In the blue intervals of size $b_{i}t$, it does not branch. The green line is at height $\sum_{i=1}^{\ell}(a_{i}+b_{i})t+ht$. 
		}
		\label{fig:picture-obstacles}
	\end{figure}
		\begin{rem}
		BBM among obstacles can also be related to F-KPP equations with a spatially inhomogeneous reaction term. Such F-KPP equations have recently been studied in e.g. \cite{DS21, HN21}. Note  that the model in Definition \ref{Def.model} corresponds to a setting where the inhomogeneity also depends on the total time horizons.
		\end{rem}
		
	\subsection{Main result}
	In this article, we derive the first order of the position of the maximal particle at time $t$, depending on $a_i, b_i, i\leq\ell$.
	
	To state the main result, we first introduce some notation.
	
	We define indices $s_{0}<s_{1}<...<s_{n}<s_{n+1}$ via $s_{0}=0$, $s_{n+1}=\ell$ and 
	\begin{equation}\label{defs}
		\{s_{1},\!...,s_{n}\}=\left\{m \in \{1, ...,\ell-1\}:\frac{\sum_{i=1}^{m}b_{i}}{\sum_{i=1}^{m}a_{i}}\geq\frac{\sum_{i=m+1}^{\ell}b_{i}}{\sum_{i=m+1}^{\ell}a_{i}}\right\}.
	\end{equation} 
	We use them to define indices $0=u_{0}^{*}<u_{1}^{*}<...<u_{n^{*}}^{*}<u_{n^{*}+1}^{*}=\ell$ iteratively.
	\begin{mydef}\label{defu} Let $u_{0}^{*}=0$. Given $u_{0}^{*},\!...,u_{\tilde{i}}^{*}$, we define $u_{\tilde{i}+1}^{*}$ as follows. We pick $\tilde{j}=\inf\{j:s_{j}>u_{\tilde{i}}^{*}\}$, the index of the next candidate. Then we pick, if it exists, $\hat{j}=\sup\{j:j\in\{\tilde{j}+1,\!...,n+1\}\text{ and }\eqref{next ui*}\}$, the largest index such that
	\begin{equation}\label{next ui*}
		\frac{\sum_{i=u_{\tilde{i}}^{*}+1}^{s_{j}}b_{i}}{\sum_{i=u_{\tilde{i}}^{*}+1}^{s_{j}}a_{i}}<\frac{\sum_{i=s_{j}+1}^{s_{\hat{j}}}b_{i}}{\sum_{i=s_{j}+1}^{s_{\hat{j}}}a_{i}}
		\text{ for all } j=\tilde{j},\!...,\hat{j}-1,
	\end{equation}
	and set $u_{\tilde{i}+1}^{*}=s_{\hat{j}}$. If such $\hat{j}$ does not exist, we set $u_{\tilde{i}+1}^{*}=s_{\tilde{j}}$. We iterate this until $s_{\tilde{j}}=\ell$ or $s_{\hat{j}}=\ell$.
	\end{mydef}
	For $m=u_{i}^{*}+1,\!...,u_{i+1}^{*}$ and $i=0,\!...,n^{*}$, we define
	\begin{align}
		\tilde{c}_{i}&=\frac{\left(\sum_{j=u_{i}^{*}+1}^{u_{i+1}^{*}}b_{j}\right)^{2}}{2\left(\sum_{j=u_{i}^{*}+1}^{u_{i+1}^{*}}a_{j}\right)^{2}}
		\quad\mbox{and}\quad f(\tilde{c}_{i})=\sqrt{\frac{1+\tilde{c}_{i}}{2}+\sqrt{\frac{\tilde{c}_{i}^{2}}{4}+\tilde{c}_{i}}}.
			\end{align}
			Moreover, let
				\begin{align}
		\label{optimal xm*_intro}x_{m}^{*}&=a_{m}f(\tilde{c}_{i})	\quad\mbox{and}\quad
		y_{m}^{*}=\frac{b_{m}}{2\frac{\sum_{j=u_{i}^{*}+1}^{u_{i+1}^{*}}a_{j}}{\sum_{j=u_{i}^{*}+1}^{u_{i+1}^{*}}b_{j}}\left(f(\tilde{c}_{i})-\frac{1}{2f(\tilde{c}_{i})}\right)}.
	\end{align}
	The main result is the following.
	\begin{theorem}\label{theorem main result}
		Let $(a_{i},b_{i})_{i=1}^{\ell}$ be some obstacle landscape such that $\sum_{i=1}^{\ell}(x_{i}^{*}+y_{i}^{*})\leq 1$. Then we have, almost surely,
		\begin{equation}
			\lim\limits_{t\to\infty}\frac{\max_{k\leq n(t)}X_{k}(t)}{t}= \sum_{i=1}^{\ell}(a_{i}+b_{i})+h^{*}
		\end{equation}
		with $	h^{*}=\sqrt{2}\left(1-\sum_{i=1}^{\ell}(x_{i}^{*}+y_{i}^{*})\right)$.
	\end{theorem}
	\begin{rem}
		Note that the suppression of branching in intervals of sizes proportional to $t$, lowers the linear order of the maximal particle position. However, the total number of particles is still of the same orderas in standard BBM, as branching is not suppressed in a neighbourhood around zero (whose size is also proportional to $t$). This is in contrast to other models in which branching is reduced, see for example \cite{chen2020branching, BH20} or where particles are absorbed \cite{BBS15,MR3077519,MR3278914,maillard2020} .
				\end{rem}
Under the additional assumption
	\begin{equation}\label{assumptionaibi_intro}
			\frac{\sum_{i=1}^{m}b_{i}}{\sum_{i=1}^{m}a_{i}}<\frac{\sum_{i=m+1}^{\ell}b_{i}}{\sum_{i=m+1}^{\ell}a_{i}} \text{ for all } m=1,\!...,\ell-1.
		\end{equation}
		the statement of Theorem \ref{theorem main result} simplifies to the following.
	\begin{cor}Assume \eqref{assumptionaibi_intro}. Then 	
		\begin{equation}\label{time to get above all obstacles under assumptionaibi_intro}
			\sum_{i=1}^{\ell}\left(x_{i}^{*}+y_{i}^{*}\right)=		\sum_{i=1}^{\ell}a_{i}f(\tilde{c}_{0})+\frac{\left(\sum_{i=1}^{\ell}b_{i}\right)^{2}}{2\sum_{i=1}^{\ell}a_{i}\left(f(\tilde{c}_{0})-\frac{1}{2f(\tilde{c}_{0})}\right)}
		\end{equation}
		and $\lim\limits_{t\to\infty}\max_{k\leq n(t)}X_{k}(t)/t$ is almost surely equal to
		\begin{equation}
				 \sum_{i=1}^{\ell}(a_{i}+b_{i})+\sqrt{2}\left(1-	\sum_{i=1}^{\ell}a_{i}f(\tilde{c}_{0})-\frac{\left(\sum_{i=1}^{\ell}b_{i}\right)^{2}}{2\sum_{i=1}^{\ell}a_{i}\left(f(\tilde{c}_{0})-\frac{1}{2f(\tilde{c}_{0})}\right)}\right).
		\end{equation}
		\end{cor}
	Note that under Assumption \eqref{assumptionaibi_intro}, we have $\{u_{0}^{*},u_{1}^{*},\!...,u_{n+1}^{*}\}=\{0,\ell\}$.
	\begin{rem}	Note that \eqref{time to get above all obstacles under assumptionaibi_intro} only depends on $\sum_{i=1}^{\ell}a_{i}$, the whole size of the $\ell$ branching areas, and $\sum_{i=1}^{\ell}b_{i}$, the whole size of all obstacles as $x_m^*$ and $y_m^*$  are proportional to $a_m$, respectively $b_m$ (at least for given $\sum_{i=1}^{\ell}a_{i}$ and $\sum_{i=1}^{\ell}b_{i}$).
	\end{rem}
	We can interpret the overall costs of the first $m$ obstacles as the ratio between their size, $\sum_{i=1}^{m}b_{i}$, and the size of the corresponding branching areas, $\sum_{i=1}^{m}a_{i}$. The same applies to the last $\ell-m$ obstacles. Then assumption \eqref{assumptionaibi_intro} says that the obstacles that a particle has already passed are always less expensive than the obstacles ahead. Hence, it will be worth it to wait for a certain amount of particles above an obstacle to cope with the more expensive remaining way and the minimal time a particle needs to get above all obstacles is 
	\begin{equation}
		\left(\sum_{i=1}^{\ell}a_{i}f(\tilde{c}_{0})+\frac{\left(\sum_{i=1}^{\ell}b_{i}\right)^{2}}{2\sum_{i=1}^{\ell}a_{i}\left(f(\tilde{c}_{0})-\frac{1}{2f(\tilde{c}_{0})}\right)}\right)t.
	\end{equation} 
	
	If Assumption \eqref{assumptionaibi_intro} does not hold, the optimal strategy requires only order one many particles above the $u_{i}^{*}$-th obstacle. Furthermore, we see that between the indices $u_{i}^{*}$ and $u_{i+1}^{*}$, the assumption 
	\begin{equation}
		\frac{\sum_{j=u_{i}^{*}+1}^{m}b_{j}}{\sum_{j=u_{i}^{*}+1}^{m}a_{j}}<\frac{\sum_{j=m+1}^{u_{i+1}^{*}}b_{j}}{\sum_{j=m+1}^{u_{i+1}^{*}}a_{j}} \text{ for all }m=u_{i}^{*}+1,\!...,u_{i+1}^{*}-1
	\end{equation}
	of \emph{late expensive obstacles} holds. Hence, we apply \eqref{time to get above all obstacles under assumptionaibi_intro} to each "block". In particular, the minimal time to get above all obstacles is 
	\begin{equation}\label{time to get above all obstacles_intro}
		\sum_{i=1}^{\ell}(x_{i}^{*}+y_{i}^{*})t=\sum_{i=0}^{n^{*}}\left(\sum_{j=u_{i}^{*}+1}^{u_{i+1}^{*}}a_{j}f(\tilde{c}_{i})+\frac{\left(\sum_{j=u_{i}^{*}+1}^{u_{i+1}^{*}}b_{j}\right)^{2}}{2\sum_{j=u_{i}^{*}+1}^{u_{i+1}^{*}}a_{j}\left(f(\tilde{c}_{i})-\frac{1}{2f(\tilde{c}_{i})}\right)}\right)t,
	\end{equation}
	the sum of the minimal times the particle needs to cross each block. The time to go through one block as fast as possible depends only on the size of all obstacles in this block, $\sum_{j=u_{i}^{*}+1}^{u_{i+1}^{*}}b_{j}$, and the size of all branching areas in this block, $\sum_{j=u_{i}^{*}+1}^{u_{i+1}^{*}}a_{j}$. Within this block, the optimal times \eqref{optimal xm*_intro} are proportional to the size of the corresponding branching area respectively obstacle. 
	\begin{rem}\label{remark no particle above obstacles}
		If \eqref{time to get above all obstacles_intro} is strictly greater than $t$, we have, almost surely,
		\begin{equation}
			\lim\limits_{t\to\infty}\frac{\max_{k\leq n(t)}X_{k}(t)}{t}< \sum_{i=1}^{\ell}(a_{i}+b_{i}).
		\end{equation}
		To identify the first order of the maximum in this case, one can proceed as follows. For $\hat{\ell}\in\{1,\!...,\ell-1\}$ and $b\in(0,b_{\hat{\ell}}]$, we define $\left(x_{1}^{*}\left(\hat{\ell},b\right),y_{1}^{*}\left(\hat{\ell},b\right),\!...,x_{\hat{\ell}}^{*}\left(\hat{\ell},b\right),y_{\hat{\ell}}^{*}\left(\hat{\ell},b\right)\right)$ for $\left(a_{1},b_{1},\!...,a_{\hat{\ell}-1},b_{\hat{\ell}-1},a_{\hat{\ell}},b\right)$ analogously to $(x_{1}^{*},y_{1}^{*},\!...,x_{\ell}^{*},y_{\ell}^{*})$ for $(a_{i},b_{i})_{i=1}^{\ell}$. Furthermore, we define $\hat{\ell}^{*}=\sup\left\{\hat{\ell}\in\{1,\!...,\ell-1\}:\sum_{i=1}^{\hat{\ell}}\left(x_{i}^{*}\left(\hat{\ell},b_{\hat{\ell}}\right)+y_{i}^{*}\left(\hat{\ell},b_{\hat{\ell}}\right)\right)\leq1\right\}$, the number of the highest obstacle that can be crossed completely until time $t$. If $\sqrt{2}\left(1-\sum_{i=1}^{\hat{\ell}^{*}}\left(x_{i}^{*}\left(\hat{\ell}^{*},b_{\hat{\ell}^{*}}\right)+y_{i}^{*}\left(\hat{\ell}^{*},b_{\hat{\ell}^{*}}\right)\right)\right)\leq a_{\hat{\ell}^{*}+1}$, we have, almost surely,
		\begin{equation}
			\lim\limits_{t\to\infty}\frac{\max_{k\leq n(t)}X_{k}(t)}{t}= \sum_{i=1}^{\hat{\ell}^{*}}(a_{i}+b_{i})+\sqrt{2}\left(1-\sum_{i=1}^{\hat{\ell}^{*}}\left(x_{i}^{*}\left(\hat{\ell}^{*},b_{\hat{\ell}^{*}}\right)+y_{i}^{*}\left(\hat{\ell}^{*},b_{\hat{\ell}^{*}}\right)\right)\right).
		\end{equation}
		If $\sqrt{2}\left(1-\sum_{i=1}^{\hat{\ell}^{*}}\left(x_{i}^{*}\left(\hat{\ell}^{*},b_{\hat{\ell}^{*}}\right)+y_{i}^{*}\left(\hat{\ell}^{*},b_{\hat{\ell}^{*}}\right)\right)\right)> a_{\hat{\ell}^{*}+1}$, we define 
		\\$b^{*}=\sup\left\{b\in \Big(0,b_{\hat{\ell}^{*}+1}\Big]:\sum_{i=1}^{\hat{\ell}^{*}+1}\left(x_{i}^{*}\left(\hat{\ell}^{*}+1,b\right)+y_{i}^{*}\left(\hat{\ell}^{*}+1,b\right)\right)\leq 1\right\}$ and have, almost surely,
		\begin{equation}
			\lim\limits_{t\to\infty}\frac{\max_{k\leq n(t)}X_{k}(t)}{t}= \sum_{i=1}^{\hat{\ell}^{*}}(a_{i}+b_{i})+a_{\hat{\ell}^{*}+1}+b^{*}.
		\end{equation}
	\end{rem}
{\bf Outline of the paper.} In Section 2, we state some preparatory lemmas, which we need later on. In Section 3, we explain how Theorem \ref{theorem main result} is connected to solving an optimization problem over possible paths and solve this optimization problem. In Section 4, we prove Theorem \ref{theorem main result}.
\section{Preparatory estimates and notation}\label{Preparatory estimates and notation} In this section, we introduce some notation, collect some Gaussian estimates and properties of standard BBM, which we need later.

	We use the following notation in the remainder. For functions $f:[0,\infty) \to \mathbb{R}$ and $g:[0,\infty) \to \mathbb{R}$, we write 
	\\$f(t)\lesssim g(t)$ if $f(t)< g(t)e^{\delta t}$, as $t \to \infty$, for all $\delta>0$, 
	\\$f(t)\gtrsim g(t)$ if $f(t)> g(t)e^{-\delta t}$, as $t \to \infty$, for all $\delta>0$
	\\and $f(t)\approx g(t)$ if $f(t)\lesssim g(t)$ and $f(t)\gtrsim g(t)$. 
	
	We need the following elementary Gaussian estimates. 
	\begin{lem}\label{lemma for obstacles}
		Let $y>0$ and $b>0$ be some constants, $X\sim \mathcal{N}(0,yt)$ and $Y\sim\mathcal{N}(0,1)$ some centered Gaussian random variables and $f:[0,\infty) \to [0,\infty)$ and $g:[0,\infty) \to [0,\infty)$ some functions with $f(t)=o(t)$ respectively $g(t)=o(t)$. Then we have
		\begin{align}
			\label{lemma for obstacles 1}\mathbb{P}\left(X\in\left[bt-g(t), bt+f(t)\right]\right)\approx e^{-\frac{b^{2}t}{2y}},
			\\\label{lemma for obstacles 2}\mathbb{P}\left(X>bt\pm f(t)\right)\lesssim e^{-\frac{b^{2}t}{2y}},
			\\\label{lemma for obstacles 3}\mathbb{P}\left(Y>t^{\frac{3}{4}}\right)\lesssim e^{-\frac{t^{\frac{3}{2}}}{2}}.
		\end{align}
	\end{lem}
	\begin{proof}
\eqref{lemma for obstacles 2} and \eqref{lemma for obstacles 3} follow immediately from Gaussian tail estimates. For \eqref{lemma for obstacles 1}, we note that the probability in the right hand side \ of \eqref{lemma for obstacles 1} is bounded from above by $	\mathbb{P}\left(X\geq bt-g(t) \right)$. Then \eqref{lemma for obstacles 1} follows again from a Gaussian tail estimate. 
	\end{proof}
	Moreover, we need an estimate on the size of the level sets of a standard binary BBM. For $x\in (0,1)$ and $a>0$, we define 
	\begin{align}
		\widehat{Z}_{a}(xt)&=\#\left\{k\leq\hat{n}(xt):\widehat{X}_{k}(xt)\geq at\right\},
		\\\widehat{Z}_{a}^{>}(xt)&=\#\left\{k\leq\hat{n}(xt):\widehat{X}_{k}(xt)\geq at\text{ and }\exists s\in[0,xt]:\widehat{X}_{k}(s)>\frac{a}{x}s+\delta t\right\},
		\\\widehat{Z}_{a}^{<}(xt)&=\#\left\{k\leq\hat{n}(xt):\widehat{X}_{k}(xt)\geq at\text{ and }\exists s\in[0,xt]:\widehat{X}_{k}(s)<\frac{a}{x}s-\delta t\right\}.
	\end{align}
	The next lemma can be essentially found in [\cite{GKS2018}, Theorem 1.1] and describes the asymptotic behaviour of $\widehat{Z}_{a}(xt)$. 
	\begin{lem}\label{lemma for branching areas}
		For any $x\in (0,1)$ and $a>0$, we have
		\begin{equation}\label{expectation branching areas}
			\mathbb{E}\left[\widehat{Z}_{a}(xt)\right]\approx \exp\left(xt-\dfrac{a^{2}t}{2x}\right).
		\end{equation}	
		For any $x\in (0,1)$ and $a \in (0,\sqrt{2}xt)$, we have, almost surely,
		\begin{equation}\label{result GKS2018}
			\lim\limits_{t\to\infty}\frac{\widehat{Z}_{a}(t)}{\mathbb{E}\left[\widehat{Z}_{a}(t)\right]}=M_{a}\quad \mbox{a.s.}
		\end{equation}
		where $M_{a}$ is the almost sure limit, as $t\to\infty$, of the McKean's martingale
		\begin{equation}
			M_{a}(t)=\sum_{k=1}^{\hat{n}(t)}\exp\left(-t\left(1+\frac{a^{2}}{2}\right)+a\widehat{X}_{k}(t)\right).
		\end{equation}
		For any $x\in (0,1)$, $a \in (0,\sqrt{2}xt)$, $\delta>0$ and $\gamma\in(0,2\delta^{2}/x)$, there exists $C_{1}>0$ such that
		\begin{equation}\label{tube restriction}
			\mathbb{P}\left(\widehat{Z}_{a}^{>}(xt)>\mathbb{E}\left[\widehat{Z}_{a}(xt)\right]e^{-\gamma t}\right)+\mathbb{P}\left(\widehat{Z}_{a}^{<}(xt)>\mathbb{E}\left[\widehat{Z}_{a}(xt)\right]e^{-\gamma t}\right)\lesssim e^{-C_{1}t}.
		\end{equation}
	\end{lem}
	\begin{proof}
		It is well known that, by the many-to-one Lemma,
		\begin{equation}
			\mathbb{E}\left[\widehat{Z}_{a}(xt)\right]=\mathbb{E}\left[\sum_{k=1}^{\hat{n}(xt)}\mathbbm{1}_{\widehat{X}_{k}(xt)\geq at}\right]\approx \exp\left(xt-\dfrac{a^{2}t}{2x}\right).
		\end{equation}	
		In [\cite{GKS2018}, Theorem 1.1], \eqref{result GKS2018} is shown. 
		
		To show \eqref{tube restriction}, we proceed analogously to the proof of [\cite{GKS2018}, Lemma 2.3]. To bound
		\begin{equation}\label{exit delta tube upper side}
			\mathbb{P}\left(\widehat{Z}_{a}^{>}(xt)>\mathbb{E}\left[\widehat{Z}_{a}(xt)\right]e^{-\gamma t}\right)
		\end{equation}
		from above via Markov´s inequality, we compute the expectation of $\widehat{Z}_{a}^{>}(xt)$. By the many-to-one-formula and distinguishing according to the position at time $xt$, we get
		\begin{align}
			\mathbb{E}\left[\widehat{Z}_{a}^{>}(xt)\right]&=e^{xt}\int_{at}^{\infty}\mathbb{P}\left(\widehat{X}_{1}(xt)\in dy\right)\mathbb{P}\left(\exists s\in [0,xt]:\widehat{X}_{1}(s)>\frac{a}{x}s+\delta t\middle| \widehat{X}_{1}(xt)=y\right)
			\\&=e^{xt}\int_{at}^{\infty}\mathbb{P}\left(\widehat{X}_{1}(xt)\in dy\right)\mathbb{P}\left(\exists s\in [0,xt]:b(s)>l(s)\middle| \widehat{X}_{1}(xt)=y\right)
		\end{align}
		with $b(s)=\widehat{X}_{1}(s)-\frac{s}{xt}\widehat{X}_{1}(xt)$ and $l(s)=(a/x-y/(xt))s+\delta t$. Since $b(s)$ is a Brownian bridge of length $xt$, we compute
		\begin{align}
			\mathbb{E}\left[\widehat{Z}_{a}^{>}(xt)\right]&=e^{xt}\int_{at}^{\infty}\mathbb{P}\left(\widehat{X}_{1}(xt)\in dy\right)\exp\left(-2\frac{l(0)l(xt)}{xt}\right)
			\\&=e^{xt}\int_{at}^{\infty}\frac{1}{\sqrt{2\pi xt}}\exp\left(-\frac{\left(y-2\delta t\right)^{2}}{2xt}\right)\exp\left(-\frac{2\delta at}{x}\right)dy.
		\end{align}
		By Lemma \ref{lemma for obstacles}, we have
		\begin{align}\label{exit tube upper side markov}
			\mathbb{E}\left[\widehat{Z}_{a}^{>}(xt)\right]
			&\approx \exp\left(xt-\frac{\left(at-2\delta t\right)^{2}}{2xt}-\frac{2\delta at}{x}\right)
			=\exp\left(xt-\frac{a^{2}t}{2x}-\frac{2\delta^{2}t}{x}\right).
		\end{align}
		Finally, using Markov´s inequality, \eqref{exit tube upper side markov} and \eqref{expectation branching areas}, we get
		\begin{equation}\label{exit tube upper side integrable}
			\mathbb{P}\left(\widehat{Z}_{a}^{>}(xt)>\mathbb{E}\left[\widehat{Z}_{a}(xt)\right]e^{-\gamma t}\right)\lesssim \exp\left(-\frac{2\delta^{2}t}{x}+\gamma t\right).
		\end{equation}
		The exponent on the r.h.s of \eqref{exit tube upper side integrable} is strictly negative for all $\gamma<2\delta^{2}/x$. Since $\mathbb{P}\left(\widehat{Z}_{a}^{>}(xt)>\mathbb{E}\left[\widehat{Z}_{a}(xt)\right]e^{-\gamma t}\right)$ can be bounded analogously, we have \eqref{tube restriction}.
	\end{proof}	
\section{An optimization problem}\label{subsection optimization problem}	
	\subsection{Optimization problem connected to Theorem \ref{theorem main result}} Our candidate for the first order of the maximum of BBM among obstacles is $\sum_{i=1}^{\ell}(a_{i}+b_{i})t+h^{*}t$, where $(h^{*})^{2}/2$ is the maximum of
	\begin{align}
		&\left(1-\sum_{i=1}^{\ell}(x_{i}+y_{i})\right)\left(\sum_{i=1}^{\ell}\left( x_{i}-\frac{a_{i}^{2}}{2x_{i}}-\frac{b_{i}^{2}}{2y_{i}}\right)+ 1-\sum_{i=1}^{\ell}(x_{i}+y_{i})\right)\nonumber
		\\\label{optimize}&=\left(1-\sum_{i=1}^{\ell}(x_{i}+y_{i})\right)\left(1-\sum_{i=1}^{\ell}\left( y_{i}+\frac{a_{i}^{2}}{2x_{i}}+\frac{b_{i}^{2}}{2y_{i}}\right)\right).
	\end{align}
	over the domain
	\begin{equation}
		D=\{(x_{1},y_{1},\!...,x_{\ell},y_{\ell})\in \mathbb{R}^{2\ell}:\eqref{domain1},\eqref{domain3},\eqref{domain4}\}
	\end{equation} 
	with conditions
	\begin{align}
		&\label{domain1}\sum_{i=1}^{m}\left( x_{i}-\frac{a_{i}^{2}}{2x_{i}}-\frac{b_{i}^{2}}{2y_{i}}\right)\geq 0 \text{ for all } m=1,\!...,\ell,
		\\\label{domain3}&\sum_{i=1}^{\ell}(x_{i}+y_{i})\leq1,
		\\\label{domain4}&x_{i}>0 \text{ and }y_{i}>0 \text{ for all } i=1,\!...,\ell.
	\end{align}
	We show that the argmax of \eqref{optimize} over $D$ equals the argmax of \eqref{optimize} over 
	\begin{equation}\hat{D}=\{(x_{1}, y_{1},\!...,x_{\ell}, y_{\ell})\in D : \eqref{domain1} \text{ holds with equality for }m=\ell\}.
	\end{equation}
	
	We briefly explain the heuristics which lead to this optimization problem. Assume that we need time $x_it$ to cross the $i$-th obstacle free area and time $y_it$ to cross the $i$-th obstacle (see Figure \ref{fig:picture-obstacles}). By Lemma \ref{lemma for branching areas}, there are approximately $\exp(x_{1}t-a_{1}^{2}t/(2x_{1}))$ particles around $a_{1}t$ at time $x_{1}t$. This implies that there are approximately $\exp(x_{1}t-a_{1}^{2}t/(2x_{1})-b_{1}^{2}t/(2y_{1}))$ offsprings of these particles at height $(a_{1}+b_{1})t$ at time $(x_{1}+y_{1})t$. Iterating this idea, this suggests that approximately $\exp(Jt)$ with
	\begin{equation}
		J=\sum_{i=1}^{\ell}\left( x_{i}-\frac{a_{i}^{2}}{2x_{i}}-\frac{b_{i}^{2}}{2y_{i}}\right)+ \left(1-\sum_{i=1}^{\ell}(x_{i}+y_{i})\right)-\frac{h^{2}}{2\left(1-\sum_{i=1}^{\ell}(x_{i}+y_{i})\right)}
	\end{equation}
	particles  spend  approximately time $x_it$ to cross the $i$-th obstacle free area and time $y_it$ to cross the $i$-th obstacle and reach height $\sum_{i=1}^\ell (a_i+b_i)t+ht$. 
	By setting $J=0$, we get, for fixed $(x_{1},y_{1},\!...,x_{\ell},y_{\ell})$, the largest possible $h$ such that there is at least one particle following the strategy. Solving for $h^{2}/2$ gives \eqref{optimize}. 
	
	Next, we find $(x_{1},y_{1},\!...,x_{\ell},y_{\ell})$ that maximize \eqref{optimize} under the additional constraint that such particles exist, which leads to the definition of the domain $D$. Condition \eqref{domain1} guarantees that our strategy has at least one particle (following the strategy) after the $m$-th obstacle at the desired time. Conditions \eqref{domain3} and \eqref{domain4} say that the total time is bounded from above by $t$ and the particles spend positive time in the $m$-th branching area respectively obstacle.	
	
	Equality in \eqref{domain1} for $m=\ell$ means there are order one many particles above all obstacles at time $\sum_{i=1}^{\ell}(x_{i}+y_{i})t$. Furthermore, \eqref{optimize} takes the form $(1-\sum_{i=1}^{\ell}(x_{i}+y_{i}))^{2}$, which is maximal if $\sum_{i=1}^{\ell}(x_{i}+y_{i})$ is minimal. That the argmax of \eqref{optimize} over $D$ equals the argmax of \eqref{optimize} over $\hat{D}$ means that the maximal particle at time $t$ is a descendant of one of the first particles above all obstacles.
	
In the remainder of this section, we solve the above optimization problem. 	We use this solution in the proof  of Theorem \ref{theorem main result} in Section \ref{section almost sure}, making the above heuristics precise.
\subsection{Optimization over $\hat D$ under Assumption \eqref{assumptionaibi_intro}}\label{Finding first particle above obstacles}
 	In this subsection, we find the argmin of $\sum_{i=1}^{\ell}(x_{i}+y_{i})$, and hence the argmax of \eqref{optimize}, over $\hat{D}$ under the additional Assumption \eqref{assumptionaibi_intro} of late expensive obstacles. 
	
The idea for finding the optimum over $\hat D$ is the following.  We assume that we have $\approx e^{c_{m-1}t}$ particles above the $(m-1)$-th obstacle at time $\sum_{i=1}^{m-1}(x_{i}+y_{i})t$. Then we choose $x_{m}$ and $y_{m}$ such that we get $\approx e^{c_{m}t}$ particles above the $m$-th obstacle as soon as possible. Afterwards, we optimize over $(c_{1},\!...,c_{\ell-1})$.
	
To formalize this, we define the following domains. For $m=1,\!...,\ell$, we define
\begin{align}
D^{m}(c_{m-1},c_{m})&=\{(x_{m},y_{m})\in\mathbb{R}^{2}:\eqref{xmc domain1}, \eqref{xmc domain3}, \eqref{xmc domain4}\},
\end{align}  
 where
	\begin{align}
		\label{xmc domain1}
		&c_{m-1}+x_{m}-\frac{a_{m}^{2}}{2x_{m}}-\frac{b_{m}^{2}}{2y_{m}}=c_{m},
		\\\label{xmc domain3}&x_{m}+y_{m}\leq N,
		\\\label{xmc domain4}&x_{m}> 0 \text{ and } y_{m}> 0,
	\end{align}
	where $N>1$ is some large constant and $c_{\ell}=c_{0}=0$. 
	\begin{rem}
		The constraints are motivated as follows: Starting with $\approx e^{c_{m-1}t}$ particles above the $(m-1)$-th obstacle and wanting $\approx e^{c_{m}t}$ particles above the $m$-th obstacle as soon as possible, means we want to minimize $x_{m}+y_{m}$ such that \eqref{xmc domain1} holds. Condition \eqref{xmc domain3} says, for technical reasons, that the total time is bounded from above by $Nt$. 
	\end{rem}
 As we let the branching Brownian motion run for a time $t$, we define 
 \begin{align}\label{deftildeD}
\widetilde{D}&=\{(x_{1},y_{1}, ...,x_{\ell},y_{\ell})\in\mathbb{R}^{2\ell}:\eqref{domain3} \text{ holds and }\nonumber \\
&\qquad\qquad(x_{m},y_{m})\in D^{m}(c_{m-1},c_{m})\text{ for }m=1, ...,\ell\}.
 \end{align}
	
	We have $c_{\ell}=0$, because \eqref{domain1} holds with equality for $m=\ell$, and $c_{0}=0$, because we start with one particle at the origin. The remaining $(c_{1},\!...,c_{\ell-1})$ should be in the domain 
	$D^{c}=\{(c_{1},\!...,c_{\ell-1})\in\mathbb{R}^{\ell-1}:\eqref{cm domain1},\eqref{cm domain2}\}$ with conditions
	\begin{align}
		\label{cm domain1}&c_{m}\geq0\text{ for all }m=1,\!...,\ell-1,
		\\\label{cm domain2}&\widetilde{D} \text{ is not empty.}
	\end{align}
	Condition \eqref{cm domain1} guarantees that there is at least $\approx 1$ particle above each obstacle. Condition \eqref{cm domain2} says that it is not impossible to find a strategy that corresponds to $(c_{1},\!...,c_{\ell-1})$.
	
	The domains $\widetilde{D}$ and $D^{c}$ are constructed such that they are compatible with the domain $\hat{D}$ in the following sense.
	\begin{lem}\label{lemma domains are compatible}
	Set, for $m=1,\!...,\ell$,
		\begin{equation}\label{compatible cm}
			c_{m}=\sum_{i=1}^{m}\left( x_{i}-\frac{a_{i}^{2}}{2x_{i}}-\frac{b_{i}^{2}}{2y_{i}}\right),
		\end{equation}
		and $c_0=0$. Then, $(x_{1},y_{1},\!...,x_{\ell},y_{\ell})\in \hat{D}$ if and only if $(c_{1},\!...,c_{\ell-1})\in D^{c}$ and $(x_{1},y_{1},\!...,x_{\ell},y_{\ell})\in \widetilde{D}$ hold.
	\end{lem}
	\begin{proof}
		Let $(x_{1},y_{1},\!...,x_{\ell},y_{\ell})\in \hat{D}$. Equality in \eqref{domain1} for $m=\ell$ implies $c_{\ell}=0$. Condition \eqref{xmc domain1} also holds by \eqref{compatible cm}, condition \eqref{xmc domain3} by \eqref{domain3} and condition \eqref{xmc domain4} by \eqref{domain4}. Hence, $(x_{m},y_{m})\in D^{m}(c_{m-1},c_{m})$ and, by \eqref{domain3}, $(x_{1},y_{1},\!...,x_{\ell},y_{\ell})\in\widetilde{D}$. Since condition \eqref{cm domain1} holds by \eqref{domain1} and $\widetilde{D}$ is not empty, we finally get $(c_{1},\!...,c_{\ell-1})\in D^{c}$. The reverse direction works analogously.
	\end{proof}
	\begin{lem}\label{lemma Dc convex}
		The domain $D^{c}$ is convex.
	\end{lem}
	\begin{proof}
		Let $c^{0}=(c^{0}_{1},\!...,c^{0}_{\ell})$ and $c^{1}=(c^{1}_{1},\!...,c^{1}_{\ell})$ be in $D^{c}$ and $\alpha\in (0,1)$. Then there exist $(x^{0}_{1},y^{0}_{1},\!...,x^{0}_{\ell},y^{0}_{\ell})$ and $(x^{1}_{1},y^{1}_{1},\!...,x^{1}_{\ell},y^{1}_{\ell})$ such that, for all $m=1,\!...,\ell$ and $q\in\{0,1\}$,
		\begin{align}
			\label{xmc0 domain1}&c^{q}_{m-1}+x^{q}_{m}-\frac{a_{m}^{2}}{2x^{q}_{m}}-\frac{b_{m}^{2}}{2y^{q}_{m}}=c^{q}_{m},
			\\\label{xmc0 domain4}&x^{q}_{m}> 0 \text{ and } y^{q}_{m}> 0,
			\\\label{xmc0 domain5}&\sum_{i=1}^{\ell}(x^{q}_{i}+y^{q}_{i})\leq1.
		\end{align}
		We have to show that $\alpha c^{0}+(1-\alpha)c^{1} \in D^{c}$. Condition \eqref{cm domain1} is clear. Let $x_{m}=\alpha x^{0}_{m}+(1-\alpha)x^{1}_{m}$. By \eqref{xmc0 domain4}, we have $x_{m}>0$ . Choosing
		\begin{equation}\label{convex combi ym}
		y_{m}=\frac{b_{m}^{2}}{2\left(\alpha (c^{0}_{m-1}-c^{0}_{m})+(1-\alpha)(c^{1}_{m-1}-c^{1}_{m})+\alpha x^{0}_{m}+(1-\alpha)x^{1}_{m}-\frac{a_{m}^{2}}{2(\alpha x^{0}_{m}+(1-\alpha)x^{1}_{m})}\right)},
		\end{equation}
		implies that
		\begin{equation}
			\alpha c^{0}_{m-1}+(1-\alpha)c^{1}_{m-1}+\alpha x^{0}_{m}+(1-\alpha)x^{1}_{m}-\frac{a_{m}^{2}}{2(\alpha x^{0}_{m}+(1-\alpha)x^{1}_{m})}-\frac{b_{m}^{2}}{2y_{m}}=\alpha c^{0}_{m}+(1-\alpha)c^{1}_{m}.
		\end{equation}
	
		Since $x-a_{m}^{2}/(2x)$ is concave in $x$, the denominator of \eqref{convex combi ym} is bounded from below by
		\begin{equation}
			2\alpha \left(c^{0}_{m-1}-c^{0}_{m}+x^{0}_{m}-\frac{a_{m}^{2}}{2x^{0}_{m}}\right)+2(1-\alpha)\left(c^{1}_{m-1}-c^{1}_{m}+x^{1}_{m}-\frac{a_{m}^{2}}{2x^{1}_{m}}\right),
		\end{equation}
		which is strictly positive by \eqref{xmc0 domain1} and \eqref{xmc0 domain4}. Hence, $y_{m}>0$. By concavity of $x-a_{m}^{2}/(2x)$ and convexity of $1/x$, we can bound $y_m$ from above
		\begin{align}
			y_{m}&\leq\frac{b_{m}^{2}}{2\left(\alpha (c^{0}_{m-1}-c^{0}_{m}+x^{0}_{m}-\frac{a_{m}^{2}}{2x^{0}_{m}})+(1-\alpha)(c^{1}_{m-1}-c^{1}_{m}+x^{1}_{m}-\frac{a_{m}^{2}}{2x^{1}_{m}})\right)}
			\\&\leq \alpha y^{0}_{m}+(1-\alpha)y^{1}_{m},
		\end{align}
	by \eqref{xmc0 domain1}.	Hence, we have
		\begin{align}
			\sum_{i=1}^{\ell}(x_{i}+y_{i})\leq\alpha\sum_{i=1}^{\ell}(x^{0}_{i}+y^{0}_{i})+(1-\alpha)\sum_{i=1}^{\ell}(x^{1}_{i}+y^{1}_{i}),
		\end{align}
		which is less or equal to $1$ by \eqref{xmc0 domain5}. 
	\end{proof}
	Proposition \ref{proposition optimal xmc existence and uniqueness} shows existence and uniqueness of the best strategy in $\widetilde{D}$ that gets $\approx e^{c_{m}t}$ particles above the $m$-th obstacle given $\approx e^{c_{m-1}t}$ particles above the $(m-1)$-th obstacle. Furthermore, it shows that the argmin satisfies some first order condition. Let $D^{m}_{x}(c_{m-1},c_{m})=\{x\in\mathbb{R}:(x,y)\in D^{m}(c_{m-1},c_{m})\text{ for some }y\in \mathbb{R}\}$.
	\begin{prop}\label{proposition optimal xmc existence and uniqueness}
		Let $(c_{1},\!...,c_{\ell-1})\in D^{c}$. Then there exists exactly one $(x_{1}^{c},y_{1}^{c},\!...,x_{\ell}^{c},y_{\ell}^{c})$ in $\widetilde{D}$ such that $(x_{m}^{c},y_{m}^{c})$ minimizes $x_{m}+y_{m}$ over $D^{m}(c_{m-1},c_{m})$ for all $m=1, ...,\ell$. 
		The component $x_{m}^{c}$ is the largest real solution of
		\begin{equation}\label{BEOx}
			\frac{b_{m}^{2}}{2\left(c_{m-1}+x_{m}^{c}-\frac{a_{m}^{2}}{2x_{m}^{c}}-c_{m}\right)^{2}}=\frac{1}{1+\frac{a_{m}^{2}}{2(x_{m}^{c})^{2}}}
		\end{equation}
		and the only solution of \eqref{BEOx} in $D^{m}_{x}(c_{m-1},c_{m})$. Moreover, $x_{m}^{c}$ is not a boundary point of $D^{m}_{x}(c_{m-1},c_{m})$.
	\end{prop}
	\begin{proof}
		Solving \eqref{xmc domain1} for $y_{m}$ gives
		\begin{equation}\label{formula for ym in minimization problem}
			y_{m}=\frac{b_{m}^{2}}{2\left(c_{m-1}+x_{m}-\frac{a_{m}^{2}}{2x_{m}}-c_{m}\right)}.
		\end{equation}
		Hence, we need to minimize
		\begin{equation}\label{plugged in formula for ym in minimization problem}
			x_{m}+\frac{b_{m}^{2}}{2\left(c_{m-1}+x_{m}-\frac{a_{m}^{2}}{2x_{m}}-c_{m}\right)}
		\end{equation}
		such that \eqref{xmc domain3} and \eqref{xmc domain4} hold. Differentiating \eqref{plugged in formula for ym in minimization problem} with respect to $x_{m}$ gives the first order condition
		\begin{equation}\label{BEOx alt}
			1-\frac{b_{m}^{2}}{2\left(c_{m-1}+x_{m}-\frac{a_{m}^{2}}{2x_{m}}-c_{m}\right)^{2}}\left(1+\frac{a_{m}^{2}}{2x_{m}^{2}}\right)=0.
		\end{equation} 
		The second derivative of \eqref{plugged in formula for ym in minimization problem} with respect to $x_{m}$ equals
		\begin{equation}\label{second derivative BEOx}
			\frac{b_{m}^{2}}{\left(c_{m-1}+x_{m}-\frac{a_{m}^{2}}{2x_{m}}-c_{m}\right)^{3}}\left(1+\frac{a_{m}^{2}}{2x_{m}^{2}}\right)^{2}
			+\frac{b_{m}^{2}}{2\left(c_{m-1}+x_{m}-\frac{a_{m}^{2}}{2x_{m}}-c_{m}\right)^{2}}\left(\frac{a_{m}^{2}}{x_{m}^{3}}\right),
		\end{equation}
	which is strictly positive. Hence, \eqref{plugged in formula for ym in minimization problem} is strictly convex in $D^{m}_{x}(c_{m-1},c_{m})$ with respect to $x_{m}$ and has a unique minimizer in the closure of $D^{m}_{x}(c_{m-1},c_{m})$.
		
		Now, suppose $x_{m}^{c} \in \partial D^{m}_{x}(c_{m-1},c_{m})$. Within the boundary, $x_{m}\to 0$ or $y_{m}\to 0$ would contradict \eqref{xmc domain1} because $x_{m}\leq N$ is not able to compensate $a_{m}^{2}/(2x_{m})\to\infty$ or $b_{m}^{2}/(2y_{m})\to\infty$. Therefore, only equality in \eqref{xmc domain3} is relevant.
		Hence, the minimum of $x_{m}+y_{m}$ over the closure of $D^{m}(c_{m-1},c_{m})$ would be $N$. But then the minimum of $\sum_{i=1}^{\ell}(x_{i}+y_{i})$ over $(x_{1},y_{1},\!...,x_{\ell},y_{\ell})$ such that $(x_{m},y_{m})\in D^{m}(c_{m-1},c_{m})$ for all $m=1, ...,\ell$ would be at least $N>1$. In this case, \eqref{domain3} could not hold. Consequently, $\widetilde{D}$ would be empty, which contradicts $(c_{1},\!...,c_{\ell-1})\in D^{c}$. 
		
		Hence, $x_{m}^{c}$ has to be in the interior of $D^{m}_{x}(c_{m-1},c_{m})$, i.e.\ $x_{m}^{c}$ is a critical point and satisfies  \eqref{BEOx alt}, which implies \eqref{BEOx}. Moreover, $\sum_{i=1}^{\ell}(x_{i}^{c}+y_{i}^{c})\leq1$ has to be true, which implies $(x_{1}^{c},y_{1}^{c},\!...,x_{\ell}^{c},y_{\ell}^{c})\in\widetilde{D}$.
		
		If $x_{m}^{c}$ was not the largest real solution of \eqref{BEOx}, we could choose $N$ so large that also this largest solution is in $D_{m}^{x}$. This would be a contradiction because, by strict convexity of \eqref{plugged in formula for ym in minimization problem}, $x_{m}^{c}$ is the only critical point in $D_{m}^{x}$.
	\end{proof}
	In Corollary \ref{corollary monotonicity xmc}, we show the monotonicity of $x_m^c$ and $x_{m+1}^{c}$ in $c_m$.   
	\begin{cor}\label{corollary monotonicity xmc}
		For all $(c_{1},\!...,c_{\ell})$ in the interior of $D^{c}$, we have $\frac{\partial }{\partial c_m}x_{m}^{c}\equiv (x_{m}^{c})^{\prime}$ exists and is $\geq 0$. Moreover, $\frac{\partial }{\partial c_m}x_{m+1}^{c}\equiv(x_{m+1}^{c})^{\prime}$ exists and is $\leq0$.
	\end{cor}
	\begin{proof}
		Let $(c_{1},\!...,c_{\ell})$ be in the interior of $D^{c}$. We defer the proof of the differentiability of $x_m$ and $x_{m+1}$ with respect to $c_m$ to Appendix A, see   Lemma \ref{lemma xmc differentiable with respect to cm}.   To prove $(x_{m}^{c})^{\prime}\geq 0$, we show that $x_{m}^{c}$ is increasing in $c_{m}$.  When increasing $c_{m}$, the l.h.s of \eqref{BEOx}   gets larger. As $x-a_{m}^{2}/(2x)$ is increasing in $x$ and $a_{m}^{2}/(2x^{2})$ is decreasing,   $x_{m}^{c}$ increases in order to satisfy \eqref{BEOx}.
		
		Similarly, we show that $x_{m+1}^{c}$ is decreasing in $c_{m}$ and hence $(x_{m+1}^{c})^{\prime}\leq0$. The first order condition for $x_{m+1}^{c}$ has the form
		\begin{equation}\label{BEOxm+1}
			\frac{b_{m+1}^{2}}{2\left(c_{m}+x_{m+1}^{c}-\frac{a_{m+1}^{2}}{2x_{m+1}^{c}}-c_{m+1}\right)^{2}}=\frac{1}{1+\frac{a_{m+1}^{2}}{2(x_{m+1}^{c})^{2}}}.
		\end{equation} 
		Now, the l.h.s of \eqref{BEOxm+1}  gets smaller when increasing $c_{m}$. As $x-a_{m+1}^{2}/(2x)$ is increasing in $x$ and $a_{m+1}^{2}/(2x^{2})$ is decreasing,  $x_{m+1}^{c}$ needs to be smaller in order to satisfy \eqref{BEOxm+1}.	
	\end{proof}	 

	\begin{prop}\label{proposition optimal times in case of late expensive obstacles}
	If $\hat{D}$ is not empty and assumption \eqref{assumptionaibi_intro} holds, then there is exactly one $(c_{1}^{*},\!...,c_{\ell-1}^{*})\in D^{c}$ that minimizes
	\begin{equation}\label{minimize whole time}
	\sum_{m=1}^{\ell}\left(x_{m}^{c}+\frac{b_{m}^{2}}{2\left(c_{m-1}+x_{m}^{c}-\frac{a_{m}^{2}}{2x_{m}^{c}}-c_{m}\right)}\right).
	\end{equation}
	This argmin is given by
		\begin{align}
	c_{m}^{*}
	\label{optimalc}
	&=\frac{\big(\sum_{i=1}^{m}a_{i}\big)\big(\sum_{i=m+1}^{\ell}b_{i}\big)-\big(\sum_{i=1}^{m}b_{i}\big)\big(\sum_{i=m+1}^{\ell}a_{i}\big)}{\sum_{i=1}^{\ell}b_{i}}\left(f(\tilde{c})-\frac{1}{2f(\tilde{c})}\right)
	\end{align}
	with
	\begin{align}
	 \tilde{c}=\frac{\left(\sum_{i=1}^{\ell}b_{i}\right)^{2}}{2\left(\sum_{i=1}^{\ell}a_{i}\right)^{2}}\quad\text{and}\quad f(\tilde{c})&=\sqrt{\dfrac{1+\tilde{c}}{2}+\sqrt{\frac{\tilde{c}^{2}}{4}+\tilde{c}}}.
	\end{align}
	The corresponding optimal times are given by 
	\begin{align}
\label{formula for xm*new}
x_{m}^{c}&=a_{m}f(\tilde{c})\quad\mbox{and} \quad
y_{m}^{c}=\frac{b_{m}}{2\frac{\sum_{i=1}^{\ell}a_{i}}{\sum_{i=1}^{\ell}b_{i}}\left(f(\tilde{c})-\frac{1}{2f(\tilde{c})}\right)},
\end{align} 
 for $m=1,\dots,\ell$.
\end{prop}
Assumption \eqref{assumptionaibi_intro} of late expensive obstacles ensures that the optimal $c_{m}^{*}$ is strictly positive. 
	Note that 
	\begin{equation}\label{f(ctilde)>1/sqrt(2)}
		f(\tilde{c})
		>\frac{1}{\sqrt{2}},
	\end{equation}
	because \eqref{f(ctilde)>1/sqrt(2)} is equivalent to $\tilde{c}+\sqrt{\tilde{c}^{2}+2\tilde{c}}>0$, which is true by $\tilde{c}>0$.
	\begin{lem}\label{lemma system of linear equations has unique solution}
		The system of linear equations
		\begin{equation}\label{linear system of equations for cm}
			c_{m}=\frac{b_{m+1}a_{m}-b_{m}a_{m+1}}{b_{m}+b_{m+1}}\left(\dfrac{x_{1}^{c}}{a_{1}}-\frac{a_{1}}{2x_{1}^{c}}\right)+\dfrac{b_{m+1}}{b_{m}+b_{m+1}}c_{m-1}+\dfrac{b_{m}}{b_{m}+b_{m+1}}c_{m+1}
		\end{equation}
		for $m=1,\!...,\ell-1$ has exactly one solution, which is given by 
		\begin{equation}\label{optimalc a1 Variante}
			c_{m}=\frac{\big(\sum_{i=1}^{m}a_{i}\big)\big(\sum_{i=m+1}^{\ell}b_{i}\big)-\big(\sum_{i=1}^{m}b_{i}\big)\big(\sum_{i=m+1}^{\ell}a_{i}\big)}{\sum_{i=1}^{\ell}b_{i}}\left(\dfrac{x_{1}^{c}}{a_{1}}-\frac{a_{1}}{2x_{1}^{c}}\right).
		\end{equation}
	\end{lem}	
	\begin{proof}
		The corresponding matrix $(a_{m,j})_{m,j=1}^{\ell-1}$, with 
		\\$a_{m,m-1}=-b_{m+1}/(b_{m}+b_{m+1})$ for $m=2,\!...,\ell-1$, 
		\\$a_{m,m}=1$ for $m=1,\!...,\ell-1$, 
		\\$a_{m,m+1}=-b_{m}/(b_{m}+b_{m+1})$ for $m=1,\!...,\ell-2$ and
		\\$a_{m,j}=0$ else, has full rank. To prove this, one notes that the $m$-th diagonal entry after Gaussian forward elimination is given by
		\begin{equation}\label{Gaussian.elim}
			\frac{b_{m}\sum_{i=1}^{m+1}b_{i}}{(b_{m}+b_{m+1})\sum_{i=1}^{m}b_{i}},
		\end{equation}
	which can be checked by induction.	As the expression \eqref{Gaussian.elim} is $\in (0,1)$ for all $m=1,\!...,\ell-1$,
		the matrix has full rank and the system of equations \eqref{linear system of equations for cm} has at most one solution. 
		Plugging \eqref{optimalc a1 Variante} into  \eqref{linear system of equations for cm}, one checks that it is indeed a solution.
	\end{proof}
	\begin{lem}\label{lemma compute xm* and ym* for cm*}
		Assume $(c_{1},\!...,c_{\ell-1})\in D^{c}$ and $c_{m}$ satisfies
		\begin{align}\label{optimalc am Variante}
		c_{m}
		=\frac{\big(\sum_{i=1}^{m}a_{i}\big)\big(\sum_{i=m+1}^{\ell}b_{i}\big)-\big(\sum_{i=1}^{m}b_{i}\big)\big(\sum_{i=m+1}^{\ell}a_{i}\big)}{\sum_{i=1}^{\ell}b_{i}}\left(\dfrac{x_{m}^{c}}{a_{m}}-\frac{a_{m}}{2x_{m}^{c}}\right),
		\end{align}
		for $m=1, ...,\ell-1$. Then we have, for $m=1,\!...,\ell$, 
		\begin{align}
			\label{formula for xm*}
			x_{m}^{c}&=a_{m}f(\tilde{c}),
			\\\label{formula for ym*}
			y_{m}^{c}&=\frac{b_{m}}{2\frac{\sum_{i=1}^{\ell}a_{i}}{\sum_{i=1}^{\ell}b_{i}}\left(f(\tilde{c})-\frac{1}{2f(\tilde{c})}\right)}
		\end{align}
		and $c_{m}$ also satisfies \eqref{optimalc}.
	\end{lem}
	\begin{proof}
Plugging \eqref{optimalc am Variante} for $c_{m}$ into  \eqref{BEOx}, the first order condition for $x_{m}$, we get
		\begin{align}\label{candidate for optimal cm in BEO xmc}
			1+\frac{a_{m}^{2}}{2\left(x_{m}^{c}\right)^{2}}=\frac{2}{b_{m}^{2}}\Bigg[\frac{\big(\sum_{i=1}^{m-1}a_{i}\big)\big(\sum_{i=m}^{\ell}b_{i}\big)-\big(\sum_{i=1}^{m-1}b_{i}\big)\big(\sum_{i=m}^{\ell}a_{i}\big)}{\sum_{i=1}^{\ell}b_{i}}
			\\+a_{m}-\frac{\big(\sum_{i=1}^{m}a_{i}\big)\big(\sum_{i=m+1}^{\ell}b_{i}\big)-\big(\sum_{i=1}^{m}b_{i}\big)\big(\sum_{i=m+1}^{\ell}a_{i}\big)}{\sum_{i=1}^{\ell}b_{i}}\Bigg]^{2}\left(\frac{x_{m}^{c}}{a_{m}}-\frac{a_{m}}{2x_{m}^{c}}\right)^{2}.
		\end{align}
	 \eqref{candidate for optimal cm in BEO xmc} is equivalent to
		\begin{equation}\label{candidate for optimal cm in BEO xmc simplified}
			1+\frac{a_{m}^{2}}{2\left(x_{m}^{c}\right)^{2}}=\frac{2\left({\sum_{i=1}^{\ell}a_{i}}\right)^{2}\left(\frac{x_{m}^{c}}{a_{m}}-\frac{a_{m}}{2x_{m}^{c}}\right)^{2}}{\left({\sum_{i=1}^{\ell}b_{i}}\right)^{2}}.
		\end{equation}
	Setting $z=x_{m}^{c}/a_{m}$ and $\tilde{c}=(b_{1}+...+b_{\ell})^{2}/(2(a_{1}+...+a_{\ell})^{2})$, we write \eqref{candidate for optimal cm in BEO xmc simplified} as
		\begin{equation}
			1+\frac{1}{2z^{2}}=\frac{1}{\tilde{c}}\left(z-\frac{1}{2z}\right)^{2},
		\end{equation}
	which is equivalent to $0=z^{4}-(1+\tilde{c})z^{2}+1/4-\tilde{c}/2$. We have
		\begin{align}
			\frac{\left(x_{m}^{c}\right)^{2}}{a_{m}^{2}}=z^{2}
			\label{formular for z^2}
			&=\dfrac{1+\tilde{c}}{2}\text{ }^{+}_{-}\sqrt{\frac{\tilde{c}^{2}}{4}+\tilde{c}}.
		\end{align}
As $x_{m}>0$, we only need to consider the positive root $\sqrt{(1+\tilde{c})/2 -\sqrt{\tilde{c}^{2}/4+\tilde{c}}}$ . If $\sqrt{(1+\tilde{c})/2 -\sqrt{\tilde{c}^{2}/4+\tilde{c}}}$ is in $D^{m}_{x}(c_{m-1},c_{m})$, this  implies that  it is  in $D^{m}_{x}(c_{m-1},c_{m})$ too, for $N$ large enough. This would lead to a contradiction as there is only one solution of \eqref{BEOx} in $D^{m}_{x}(c_{m-1},c_{m})$ by Proposition \ref{proposition optimal xmc existence and uniqueness}. Hence, we have \eqref{formula for xm*}, i.e. $x_{m}^{c}=a_{m}z=a_{m}f(\tilde{c})$ with
		\begin{equation}
			f(\tilde{c})=\sqrt{\dfrac{1+\tilde{c}}{2}+\sqrt{\frac{\tilde{c}^{2}}{4}+\tilde{c}}}.
		\end{equation}
		Plugging \eqref{optimalc am Variante} into \eqref{formula for ym in minimization problem}, we get
		\begin{align}
			y_{m}^{c}
			=&\frac{b_{m}^{2}}{2}\times\Bigg[ \frac{\big(\sum_{i=1}^{m-1}a_{i}\big)\big(\sum_{i=m}^{\ell}b_{i}\big)-\big(\sum_{i=1}^{m-1}b_{i}\big)\big(\sum_{i=m}^{\ell}a_{i}\big)}{\sum_{i=1}^{\ell}b_{i}} +a_{m}
			\\\nonumber&-\frac{\big(\sum_{i=1}^{m}a_{i}\big)\big(\sum_{i=m+1}^{\ell}b_{i}\big)-\big(\sum_{i=1}^{m}b_{i}\big)\big(\sum_{i=m+1}^{\ell}a_{i}\big)}{\sum_{i=1}^{\ell}b_{i}} \Bigg]^{-1}\left(\frac{x_{m}^{c}}{a_{m}}-\frac{a_{m}}{2x_{m}^{c}}\right)^{-1}.
		\end{align}
		If we cancel all terms that arise with both signs, simplify the fraction and use \eqref{formula for xm*}, we get \eqref{formula for ym*}. That $c_{m}$ satisfies \eqref{optimalc}, follows by plugging \eqref{formula for xm*} into \eqref{optimalc am Variante}.
	\end{proof}

	\begin{proof}[Proof of Proposition \ref{proposition optimal times in case of late expensive obstacles} ]
		Assume that the interior of $D^{c}$ is not empty. (At the end of this proof, we will justify that this assumption does not cause a loss of generality.) Taking the derivative of \eqref{minimize whole time} with respect to $c_{m}$ gives the first order condition
		\begin{multline}
			\label{BEOcalt}\left(x_{m}^{c}\right)^{\prime}-\frac{b_{m}^{2}}{2\left(c_{m-1}+x_{m}^{c}-\frac{a_{m}^{2}}{2x_{m}^{c}}-c_{m}\right)^{2}}\left(	\left(x_{m}^{c}\right)^{\prime}+\frac{a_{m}^{2}}{2	\left(x_{m}^{c}\right)^{2}}	\left(x_{m}^{c}\right)^{\prime}-1\right)+\left(x_{m+1}^{c}\right)^{\prime}
			\\-\frac{b_{m+1}^{2}}{2\left(c_{m}+x_{m+1}^{c}-\frac{a_{m+1}^{2}}{2x_{m+1}^{c}}-c_{m+1}\right)^{2}}\left(1+	\left(x_{m+1}^{c}\right)^{\prime}+\frac{a_{m+1}^{2}}{2	\left(x_{m+1}^{c}\right)^{2}}	\left(x_{m+1}^{c}\right)^{\prime}\right)=0.
		\end{multline}
		By plugging \eqref{BEOx} into \eqref{BEOcalt}, we obtain
		\begin{align}
			0
			\label{BEOc transformed}
			=&\frac{1}{1+\frac{a_{m}^{2}}{2\left(x_{m}^{c}\right)^{2}}}-\frac{1}{1+\frac{a_{m+1}^{2}}{2\left(x_{m+1}^{c}\right)^{2}}}.
		\end{align}
		We note that \eqref{BEOc transformed} holds if and only if
		\begin{equation}\label{BEOc}
			\frac{a_{m}}{x_{m}^{c}}=\frac{a_{m+1}}{x_{m+1}^{c}}.
		\end{equation}
		The second derivative of \eqref{minimize whole time} with respect to $c_{m}$ equals
		\begin{equation}
			\frac{1}{\left(1+\frac{a_{m}^{2}}{2\left(x_{m}^{c}\right)^{2}}\right)^{2}}\frac{a_{m}^{2}}{\left(x_{m}^{c}\right)^{3}}\left(x_{m}^{c}\right)^{\prime}-\frac{1}{\left(1+\frac{a_{m+1}^{2}}{2\left(x_{m+1}^{c}\right)^{2}}\right)^{2}}\frac{a_{m+1}^{2}}{\left(x_{m+1}^{c}\right)^{3}}\left(x_{m+1}^{c}\right)^{\prime}
		\end{equation}
		and is non-negative, because $\left(x_{m}^{c}\right)^{\prime}\geq0$ and $\left(x_{m+1}^{c}\right)^{\prime}\leq0$ by Corollary \ref{corollary monotonicity xmc}. Hence, \eqref{minimize whole time} is weakly convex with respect to $c_{m}$ in the relevant domain $D^{c}$. We will see soon that \eqref{minimize whole time} has exactly one critical point in $D^{c}$. Because of this uniqueness, the weak convexity of \eqref{minimize whole time} and Lemmas \ref{lemma Dc convex} and \ref{lemma xm(C) continuous and simple}, the desired minimum has to be attained at this critical point.
		
		If we combine \eqref{BEOx} and \eqref{BEOc}, we see that the critical $(c_{1},\!...,c_{\ell-1})$ has to satisfy
		\begin{equation}
			\frac{b_{m}}{c_{m-1}+x_{m}^{c}-\frac{a_{m}^{2}}{2x_{m}^{c}}-c_{m}}=\frac{b_{m+1}}{c_{m}+x_{m+1}^{c}-\frac{a_{m+1}^{2}}{2x_{m+1}^{c}}-c_{m+1}}.
		\end{equation}
		We solve this for $c_{m}$ and use \eqref{BEOc} again to get the system of linear equations
		\begin{align}
			c_{m}=&\frac{b_{m+1}\left(c_{m-1}+x_{m}^{c}-\frac{a_{m}^{2}}{2x_{m}^{c}}\right)-b_{m}\left(x_{m+1}^{c}-\frac{a_{m+1}^{2}}{2x_{m+1}^{c}}-c_{m+1}\right)}{b_{m}+b_{m+1}}
			\\\label{linear system of equations again}
			=&\frac{b_{m+1}a_{m}-b_{m}a_{m+1}}{b_{m}+b_{m+1}}\left(\dfrac{x_{1}^{c}}{a_{1}}-\frac{a_{1}}{2x_{1}^{c}}\right)+\dfrac{b_{m+1}}{b_{m}+b_{m+1}}c_{m-1}+\dfrac{b_{m}}{b_{m}+b_{m+1}}c_{m+1}
		\end{align}
		for $m=1,\!...,\ell-1$. By Lemma \ref{lemma system of linear equations has unique solution}, the unique solution of \eqref{linear system of equations again} is given by \eqref{optimalc a1 Variante}. By \eqref{BEOc}, it also satisfies \eqref{optimalc am Variante}. Hence, our candidates for the optimal times are given in Lemma \ref{lemma compute xm* and ym* for cm*} with $c_{m}$ satisfying \eqref{optimalc}.
		
		It remains to show that the candidate given by \eqref{optimalc} is in $D^{c}$. First, $c_{m}\geq0$ follows by \eqref{assumptionaibi_intro} and \eqref{f(ctilde)>1/sqrt(2)}. We prove that $\tilde{D}$ is not empty by showing that the candidate satisfying \eqref{formula for xm*} and \eqref{formula for ym*} is an element of $\tilde{D}$. Condition \eqref{xmc domain1} holds by construction of $y_{m}^{c}$ via \eqref{formula for ym in minimization problem}. Condition \eqref{xmc domain4} follows by \eqref{f(ctilde)>1/sqrt(2)}. To get \eqref{xmc domain3} and in particular \eqref{domain3}, we need
		\begin{equation}\label{candidate optimal times<=1}
			\sum_{i=1}^{\ell}(x_{i}^{c}+y_{i}^{c})=\sum_{i=1}^{\ell}a_{i}f(\tilde{c})+\frac{\left(\sum_{i=1}^{\ell}b_{i}\right)^{2}}{2\sum_{i=1}^{\ell}a_{i}\left(f(\tilde{c})-\frac{1}{2f(\tilde{c})}\right)}\leq1.
		\end{equation}
		If we already knew that \eqref{candidate optimal times<=1} was true, we could argue as follows: The candidate given by \eqref{optimalc} is in $D^{c}$ with optimal times \eqref{formula for xm*} and \eqref{formula for ym*}. Furthermore, it is a critical point, because it satisfies \eqref{BEOc} by \eqref{formula for xm*}. As already mentioned, uniqueness of this critical point and weak convexity of \eqref{minimize whole time} imply that the argmin of \eqref{minimize whole time} over $D^{c}$ is given by \eqref{optimalc}.
		
		We show that \eqref{candidate optimal times<=1} is indeed true. Therefore, we define the domain $\hat{D}^{n}=\{(x_{1}, y_{1},\!...,x_{\ell}, y_{\ell})\in \mathbb{R}^{2\ell} : \eqref{domain1},\sum_{i=1}^{\ell}(x_{i}+y_{i})\leq n,\eqref{domain4}, \eqref{domain1} \text{ holds with equality for }m=\ell\}$. It is almost the same definition as for $\hat{D}$ but the total time is bounded from above by $nt$ instead of $t$. We choose $n>1$ so large that the candidate on the l.h.s of \eqref{candidate optimal times<=1} is smaller than $n$ and $\hat{D}^{n}$ is not empty. (For example, choose $n$ as the maximum of the l.h.s of \eqref{candidate optimal times<=1} and $\sum_{i=1}^{\ell}(2a_{i}+b_{i}^{2}/(3.5a_{i}))$. Then $\hat{D}^{n}$ is not empty because it contains $(2a_{1},b_{1}^{2}/(3.5a_{1}),\!...,2a_{\ell},b_{\ell}^{2}/(3.5a_{\ell}))$.)

		Then all entire results in Subsection \ref{Finding first particle above obstacles} carry over if we adapt the definition of $\widetilde{D}$ and ensure $N>n$ in \eqref{xmc domain3}. Hence, the minimum of $\sum_{i=1}^{\ell}(x_{i}+y_{i})$ over $\hat{D}^{n}$ is the l.h.s of \eqref{candidate optimal times<=1}. Since $\hat{D}$ is a subset of $\hat{D}^{n}$, we have
		\begin{multline}
			\min\left\{\sum_{i=1}^{\ell}(x_{i}+y_{i}):(x_{1}, y_{1},\!...,x_{\ell}, y_{\ell})\in\hat{D}\right\}\\\geq\min\left\{\sum_{i=1}^{\ell}(x_{i}+y_{i}):(x_{1}, y_{1},\!...,x_{\ell}, y_{\ell})\in\hat{D}^{n}\right\}.
		\end{multline}
		By \eqref{domain3}, this implies that $\hat{D}$ is not empty if and only if \eqref{candidate optimal times<=1} is true. Since we assumed non-emptiness of $\hat{D}$, the claim follows. (Looking at $\hat{D}^{n}$ also justifies that we can assume non emptiness of the interior of $D^{c}$ without a loss of generality.)
	\end{proof}
	\begin{cor}\label{corolarry D not empty in late expensive obstacle case}
		For any obstacle landscape $(a_{i},{b}_{i})_{i=1}^{\ell}$ satisfying assumption \eqref{assumptionaibi_intro}, the domain $\hat{D}$ is not empty if and only if 
		\begin{equation}
			\sum_{i=1}^{\ell}a_{i}f(\tilde{c})+\frac{\left(\sum_{i=1}^{\ell}b_{i}\right)^{2}}{2\sum_{i=1}^{\ell}a_{i}\left(f(\tilde{c})-\frac{1}{2f(\tilde{c})}\right)}\leq1.
		\end{equation}
	\end{cor}
	The main result of this subsection is the following.
	\begin{theorem}\label{theorem candidate in case of late expensive obstacles}
		If $\hat{D}$ is not empty and assumption \eqref{assumptionaibi_intro} holds, then the maximum of \eqref{optimize} over $\hat{D}$ equals $(h^{*})^{2}/2$ with
		\begin{align}
			h^{*}&=\sqrt{2}\left(1-\sum_{i=1}^{\ell}(x_{i}^{*}+y_{i}^{*})\right)
			\\&=\sqrt{2}\left(1-\sum_{i=1}^{\ell}a_{i}f(\tilde{c})-\frac{\left(\sum_{i=1}^{\ell}b_{i}\right)^{2}}{2\sum_{i=1}^{\ell}a_{i}\left(f(\tilde{c})-\frac{1}{2f(\tilde{c})}\right)}\right).
		\end{align}
	\end{theorem}
	\begin{proof}
		The argmax of \eqref{optimize} over $\hat{D}$ equals the argmin of $\sum_{i=1}^{\ell}(x_{i}+y_{i})$ over $\hat{D}$. By Lemma \ref{lemma domains are compatible} and Proposition \ref{proposition optimal xmc existence and uniqueness}, this is equivalent to minimizing $\sum_{i=1}^{\ell}(x_{i}^{c}+y_{i}^{c})$ over $D^{c}$. By Proposition \ref{proposition optimal times in case of late expensive obstacles}, this argmin consists of the optimal times in Lemma \ref{lemma compute xm* and ym* for cm*}.
	\end{proof}
\subsection{Optimization over $\hat D$}\label{Optimal division into admissible blocks}
	In this subsection, we find the argmin of $\sum_{i=1}^{\ell}(x_{i}+y_{i})$, and hence the argmax of \eqref{optimize}, over $\hat{D}$.
	
	We introduce  the shorthand notation
	\begin{equation}
		\sum\nolimits_{m}^{n}=\frac{\sum_{i=m}^{n}b_{i}}{\sum_{i=m}^{n}a_{i}}.
	\end{equation}
We divide the obstacle landscape into blocks and  need the following definitions. 
	\begin{mydef}
		We call a sequence of natural numbers $0=u_{0}<u_{1}<...<u_{n_{1}}<u_{n_{1}+1}=\ell$ an admissible division into blocks if
		\begin{equation}\label{definition admissible division into blocks}
			\sum\nolimits_{u_{i}+1}^{m}<\sum\nolimits_{m+1}^{u_{i+1}}
		\end{equation} 
		for all $i=0,\!...,n_{1}$ and $m=u_{i}+1,\!...,u_{i+1}-1$.
	\end{mydef}
	\begin{rem}
		 \eqref{definition admissible division into blocks} says that the obstacle landscape $(a_{j},b_{j})_{j={u_{i}+1}}^{u_{i+1}}$ satisfies Assumption \eqref{assumptionaibi_intro}.
	\end{rem}
	\begin{mydef}
		For two admissible divisions into blocks, $0=u_{0}<u_{1}<...<u_{n_{1}}<u_{n_{1}+1}=\ell$ and $0=v_{0}<v_{1}<...<v_{n_{2}}<v_{n_{2}+1}=\ell$, we define their intersection 
		\begin{equation}
	\{w_{0},w_{1},\!...,w_{n_{3}+1}\}=\{u_{0},u_{1},\!...,u_{n_{1}+1}\} \cap \{v_{0},v_{1},\!...,v_{n_{2}+1}\},
		\end{equation}
		where $0=w_{0}<w_{1}<...<w_{n_{3}}<w_{n_{3}+1}=\ell$. 
	\end{mydef}
	For any natural numbers $0=u_{0}<u_{1}<...<u_{n_{1}}<u_{n_{1}+1}=\ell$, we define the domain
	\begin{equation}
		 D^{c}_{u_{1},\!...,u_{n_{1}}}=\left\{(c_{1},\!...,c_{\ell-1})\in D^{c}:c_{m}=0\text{ for all }m \in \{u_{1},\!...,u_{n_{1}}\}\right\}.
	\end{equation}
	The main result of this subsection is the following.
	\begin{theorem}\label{theorem candidate in the general case Dhat Variante}
		If $\hat{D}$ is not empty, the maximum of \eqref{optimize} over $\hat{D}$ equals $(h^{*})^{2}/2$ with
		\begin{align}\label{h* in the general case Dhat Variante}
			h^{*}=\sqrt{2}\left(1-\sum_{i=0}^{n^{*}}\left(\sum_{j=u_{i}^{*}+1}^{u_{i+1}^{*}}a_{j}f(\tilde{c}_{i})+\frac{\left(\sum_{j=u_{i}^{*}+1}^{u_{i+1}^{*}}b_{j}\right)^{2}}{2\sum_{j=u_{i}^{*}+1}^{u_{i+1}^{*}}a_{j}\left(f(\tilde{c}_{i})-\frac{1}{2f(\tilde{c}_{i})}\right)}\right)\right).
		\end{align}
	\end{theorem}
	\begin{proof}
		The argmax of \eqref{optimize} over $\hat{D}$ equals the argmin of $\sum_{i=1}^{\ell}(x_{i}+y_{i})$ over $\hat{D}$. By Lemma \ref{lemma domains are compatible} and Proposition \ref{proposition optimal xmc existence and uniqueness}, this is equivalent to minimizing $\sum_{i=1}^{\ell}(x_{i}^{c}+y_{i}^{c})$ over $D^{c}$. By Theorem \ref{theorem candidate in case of late expensive obstacles}, we already know that \eqref{h* in the general case Dhat Variante} is optimal if assumption \eqref{assumptionaibi_intro} holds. 
		
		If assumption \eqref{assumptionaibi_intro} does not hold, we can apply the following lemma, which we will prove later.
		\begin{lem}\label{lemma cuhat=0}
			If \eqref{assumptionaibi_intro} does not hold, there is some $\hat{u}\in\{1,\!...,\ell-1\}$ such that the argmin of $\sum_{i=1}^{\ell}(x_{i}^{c}+y_{i}^{c})$ over $D^{c}$ equals the argmin of $\sum_{i=1}^{\ell}(x_{i}^{c}+y_{i}^{c})$ over $D^{c}_{\hat{u}}$.
		\end{lem} 
		 If $0<\hat{u}<\ell$ is not admissible, we can apply Lemma \ref{lemma cuhat=0} again. Consequently, there has to be some admissible division into blocks $0=\hat{u}_{0}<\hat{u}_{1}<...<\hat{u}_{\hat{n}}<\hat{u}_{\hat{n}+1}=\ell$ such that $c_{m}^{*}=0$ if and only if $m$ is in the set $\{\hat{u}_{1},\!...,\hat{u}_{\hat{n}}\}$. We want to show that $0=\hat{u}_{0}<\hat{u}_{1}<...<\hat{u}_{\hat{n}}<\hat{u}_{\hat{n}+1}=\ell$ is given by $0=u_{0}^{*}<u_{1}^{*}<...<u_{n^{*}}^{*}<u_{n^{*}+1}^{*}=\ell$, which was introduced in Definition  \ref{defu}. We need the following lemmas, which we will prove later.
		\begin{lem}\label{lemma intersection of all admissible is optimal}
			The intersection of two admissible divisions into blocks is admissible.	
		\end{lem}
		\begin{lem}\label{lemma intersection is better}
			Let $0=u_{0}<u_{1}<...<u_{n_{1}}<u_{n_{1}+1}=\ell$ and $0=v_{0}<v_{1}<...<v_{n_{2}}<v_{n_{2}+1}=\ell$ be two admissible divisions into blocks such that $D^{c}_{u_{1},\!...,u_{n_{1}}}$ and $D^{c}_{v_{1},\!...,v_{n_{2}}}$ are not empty. Let $0=w_{0}<w_{1}<...<w_{n_{3}}<w_{n_{3}+1}=\ell$ be their intersection. Assume that there are $\hat{i}\in\{1,\!...,n_{1}\}$ and $\hat{j}\in\{1,\!...,n_{2}\}$ such that $u_{\hat{i}}\notin \{w_{1},\!...,w_{n_{3}}\}$ and $v_{\hat{j}}\notin \{w_{1},\!...,w_{n_{3}}\}$. Then we have
			\begin{align}
				&\min\left\{\sum_{i=1}^{\ell}(x_{i}^{c}+y_{i}^{c}):(c_{1},\!...,c_{\ell-1})\in D^{c}_{u_{1},\!...,u_{n_{1}}}\cup D^{c}_{v_{1},\!...,v_{n_{2}}}\right\}\nonumber\\
				\label{equation intersection better v}
				&>\min\left\{\sum_{i=1}^{\ell}(x_{i}^{c}+y_{i}^{c}):(c_{1},\!...,c_{\ell-1})\in D^{c}_{w_{1},\!...,w_{n_{3}}}\right\},
			\end{align}
		\end{lem}
		\begin{lem}\label{lemma optimal is a subset of indices for which assumption fails} Recall the definition of $s_{0},s_{1},...,s_{n+1}$ in \eqref{defs}.
			The division into blocks $0=s_{0}<s_{1}<...<s_{n}<s_{n+1}=\ell$  is admissible.
		\end{lem}
		\begin{lem}\label{lemma sufficent to check inequality for si}
			If there exist some $\tilde{i}\in\{1,\!...,n\}$ and $\tilde{j}\in\{\tilde{i}+2,\!...,n\}$ such that
			\begin{equation}\label{< for indices for which assumption fails}
				\sum\nolimits_{s_{\tilde{i}}+1}^{s_{i}}<\sum\nolimits_{s_{i}+1}^{s_{\tilde{j}}}\text{ for all }i=\tilde{i}+1,\!...,\tilde{j}-1,
			\end{equation}
			then also
			\begin{equation}\label{< also for indices between them}
				\sum\nolimits_{s_{\tilde{i}}+1}^{m}<\sum\nolimits_{m+1}^{s_{\tilde{j}}}\text{ for all } m=s_{\tilde{i}}+1,\!...,s_{\tilde{j}}-1.
			\end{equation} 
		\end{lem}
		By Lemma \ref{lemma intersection of all admissible is optimal} and Lemma \ref{lemma intersection is better}, the optimal admissible division into blocks $0=\hat{u}_{0}<\hat{u}_{1}<...<\hat{u}_{\hat{n}}<\hat{u}_{\hat{n}+1}=\ell$ is unique and the intersection of all admissible divisions into blocks. Furthermore, $\{\hat{u}_{1},\!...,\hat{u}_{\hat{n}}\}$ is a subset of $\{s_{1},\!...,s_{n}\}$ because $0=s_{0}<s_{1}<...<s_{n}<s_{n+1}=\ell$ is admissible by Lemma \ref{lemma optimal is a subset of indices for which assumption fails}.
		
		We identify $0=\hat{u}_{0}<\hat{u}_{1}<...<\hat{u}_{\hat{n}}<\hat{u}_{\hat{n}+1}=\ell$ by induction. If we already know $\hat{u}_{0},\!...,\hat{u}_{\tilde{i}}$, we find $\hat{u}_{\tilde{i}+1}$ as follows. We pick $\tilde{j}=\inf\{j:s_{j}>\hat{u}_{\tilde{i}}\}$, the index of the next candidate. By Lemma \ref{lemma intersection of all admissible is optimal} and Lemma \ref{lemma optimal is a subset of indices for which assumption fails}, we know that $s_{\tilde{j}}\notin\{\hat{u}_{0},\!...,\hat{u}_{\hat{n}+1}\}$ if and only if there exist $i_{1}<\tilde{j}<i_{2}$ such that
		\begin{equation}\label{long condition}
			\sum\nolimits_{s_{i_{1}}+1}^{m}<\sum\nolimits_{m+1}^{s_{i_{2}}} \text{ for all } m=s_{i_{1}}+1,\!...,s_{i_{2}}-1.
		\end{equation}
		Furthermore, we know $s_{i_{1}}=\hat{u}_{\tilde{i}}$ because otherwise $\hat{u}_{\tilde{i}}\notin\{\hat{u}_{0},\!...,\hat{u}_{\hat{n}+1}\}$. If $i_{2}$ does not exist, we have $s_{\tilde{j}}=\hat{u}_{\tilde{i}+1}$. If it exists, we pick the largest one, $\hat{j}=\sup\{i_{2}:i_{2}\in\{\tilde{j}+1,\!...,n+1\} \text{ and }\eqref{long condition}\}$, and have $s_{\hat{j}}=\hat{u}_{\tilde{i}+1}$. (By definition of $\hat{j}$, we have $s_{\hat{j}}\in\{\hat{u}_{0},\!...,\hat{u}_{\hat{n}+1}\}$ and $s_{j}\notin\{\hat{u}_{0},\!...,\hat{u}_{\hat{n}+1}\}$ for all $j=\tilde{j},\!...,\hat{j}-1$.) We iterate this until $s_{\tilde{j}}=\ell$ or $s_{\hat{j}}=\ell$. By Lemma \ref{lemma sufficent to check inequality for si}, it is sufficient to check \eqref{long condition} only for $m\in \{s_{1},\!...,s_{n}\}$. Hence, the optimal division into blocks $0=\hat{u}_{0}<\hat{u}_{1}<...<\hat{u}_{\hat{n}}<\hat{u}_{\hat{n}+1}=\ell$ is given by $0=u_{0}^{*}<u_{1}^{*}<...<u_{n^{*}}^{*}<u_{n^{*}+1}^{*}=\ell$.
		
		By Theorem \ref{theorem candidate in case of late expensive obstacles} and the admissibility of $0=u_{0}^{*}<u_{1}^{*}<...<u_{n^{*}}^{*}<u_{n^{*}+1}^{*}=\ell$, the optimal times are
		\begin{align}
			&x_{m}^{*}=a_{m}f(\tilde{c}_{i}),\quad\mbox{and}\quad
			y_{m}^{*}=\frac{b_{m}}{2\frac{\sum_{j=u_{i}^{*}+1}^{u_{i+1}^{*}}a_{j}}{\sum_{j=u_{i}^{*}+1}^{u_{i+1}^{*}}b_{j}}\left(f(\tilde{c}_{i})-\frac{1}{2f(\tilde{c}_{i})}\right)}
			\\&\text{with } f(\tilde{c}_{i})=\sqrt{\frac{1+\tilde{c}_{i}}{2}+\sqrt{\frac{\tilde{c}_{i}^{2}}{4}+\tilde{c}_{i}}}
			\quad\text{and }\quad\tilde{c}_{i}=\frac{\left(\sum_{j=u_{i}^{*}+1}^{u_{i+1}^{*}}b_{j}\right)^{2}}{2\left(\sum_{j=u_{i}^{*}+1}^{u_{i+1}^{*}}a_{j}\right)^{2}}
		\end{align}
		for all $m=u_{i}^{*}+1,\!...,u_{i+1}^{*}$ and $i=0,\!...,n^{*}$. Hence, the maximum of \eqref{optimize} over $\hat{D}$ equals $(h^{*})^{2}/2$ where $h^{*}=\sqrt{2}(1-\sum_{i=1}^{\ell}(x_{i}^{*}+y_{i}^{*}))$ is equal to \eqref{h* in the general case Dhat Variante}.
	\end{proof}	
	\begin{cor}\label{corolarry D not empty Dhat Variante}
		For any obstacle landscape $(a_{i},{b}_{i})_{i=1}^{\ell}$, the domain $\hat{D}$ is not empty if and only if 
		\begin{equation}
			\sum_{i=0}^{n^{*}}\left(\sum_{j=u_{i}^{*}+1}^{u_{i+1}^{*}}a_{j}f(\tilde{c}_{i})+\frac{\left(\sum_{j=u_{i}^{*}+1}^{u_{i+1}^{*}}b_{j}\right)^{2}}{2\sum_{j=u_{i}^{*}+1}^{u_{i+1}^{*}}a_{j}\left(f(\tilde{c}_{i})-\frac{1}{2f(\tilde{c}_{i})}\right)}\right)\leq1.
		\end{equation}
	\end{cor}
	\begin{proof}
		The claim follows analogously to Corollary \ref{corolarry D not empty in late expensive obstacle case} by looking at $\hat{D}^{n}$.
	\end{proof}
	It remains to prove the five lemmas that we used in the proof of Theorem \ref{theorem candidate in the general case}. 
	\begin{proof}[Proof of Lemma \ref{lemma cuhat=0}]
		First, we show that the argmin has to be in the boundary of $D^{c}$ if \eqref{assumptionaibi_intro} does not hold. Recall the definitions of
		$D^{c}=\{(c_{1},\!...,c_{\ell-1})\in\mathbb{R}^{\ell-1}:\eqref{cm domain1},\eqref{cm domain2}\}$ 
	and  $\widetilde{D}$, which is defined in \eqref{deftildeD}.
	 In the proof of Proposition \ref{proposition optimal times in case of late expensive obstacles}, we saw that a critical point would have to satisfy \eqref{linear system of equations again}. By Lemma \ref{lemma system of linear equations has unique solution}, the unique solution of \eqref{linear system of equations again} is given by \eqref{optimalc}. But \eqref{optimalc} is not strictly positive if assumption \eqref{assumptionaibi_intro} does not hold. Hence, by condition \eqref{cm domain1}, there is no critical point in the interior of $D^{c}$ and the argmin has to be in the boundary.
		
		Next, we show that there is some $\hat{u}\in\{1,\!...,\ell-1\}$ such that $c_{\hat{u}}^{*}=0$. Assume w.l.o.g. that the minimum of $\sum_{i=1}^{\ell}(x_{i}^{c}+y_{i}^{c})$ over $D^{c}$ is strictly smaller than $1$.\footnote{If the minimum of $\sum_{i=1}^{\ell}(x_{i}^{c}+y_{i}^{c})$ over $D^{c}$ is equal to $1$, one looks at $\hat{D}^{n}$ instead and still gets the existence of $\hat{u}$.} Let $(\hat{c}_{1},\!...,\hat{c}_{\ell-1})$ be in the boundary of $D^{c}$ such that $\hat{c}_{m}>0$ for $m=1,\!...,\ell-1$. Suppose that $(\hat{c}_{1},\!...,\hat{c}_{\ell-1})$ is the argmin of $\sum_{i=1}^{\ell}(x_{i}^{c}+y_{i}^{c})$ over $D^{c}$. For $m=1,...,\ell$, let $(x_{m}^{\hat{c}},y_{m}^{\hat{c}})$ be the argmin of $x_{m}+y_{m}$ over $D^{m}(\hat{c}_{m-1},\hat{c}_{m})$.
		
		Note that $D^{m}_{x}(\hat{c}_{m-1},\hat{c}_{m})$ is a closed interval by Lemma \ref{lemma Dmx(C) closed interval}. By Proposition \ref{proposition optimal xmc existence and uniqueness}, there is some $\epsilon>0$ such that $\{x_{m}:|x_{m}-x_{m}^{\hat{c}}|<4\epsilon \}\subset D^{m}_{x}(\hat{c}_{m-1},\hat{c}_{m})$ for $m=1,\!...,\ell$. Since the minimum of $\sum_{i=1}^{\ell}(x_{i}^{c}+y_{i}^{c})$ over $D^{c}$ is strictly smaller than $1$, we can choose $\epsilon>0$ so small that also $\sum_{m=1}^{\ell}(x_{m}+y_{m})<1$ for all $(x_{m},y_{m})\in D^{m}(\hat{c}_{m-1},\hat{c}_{m})$ with $x_{m}\in\{x_{m}:|x_{m}-x_{m}^{\hat{c}}|<\epsilon \}$. By Lemma \ref{lemma Dmx(C) continuous in C} and its analogue for $c_{m-1}$, there exists $\delta>0$ such that for all $(c_{1},\!...,c_{\ell-1})$ satisfying $|c_{m}-\hat{c}_{m}|<\delta$ for $m=1,\!...,\ell-1$, we have $\{x_{m}:|x_{m}-x_{m}^{\hat{c}}|<\epsilon \}\subset D^{m}_{x}(c_{m-1},c_{m})$. Since $\hat{c}_{m}>0$, we can choose $\delta>0$ so small that also $c_{m}>0$ for all $|c_{m}-\hat{c}_{m}|<\delta$. This is a contradiction because we supposed that $(\hat{c}_{1},\!...,\hat{c}_{\ell-1})$ is in the boundary of $D^{c}$. Hence, $(\hat{c}_{1},\!...,\hat{c}_{\ell-1})$ can not be optimal and there has to be some $\hat{u}$ such that $c_{\hat{u}}^{*}=0$.
	\end{proof}
	For the other four lemmas, we use the following simple implications. For any $\hat{a}_{i}, \hat{b}_{i}>0$, we have
	\begin{align}
		\label{A}&\left[\frac{\hat{b}_{1}}{\hat{a}_{1}}<\frac{\hat{b}_{2}+\hat{b}_{3}}{\hat{a}_{2}+\hat{a}_{3}} \text{ and } \frac{\hat{b}_{2}}{\hat{a}_{2}}<\frac{\hat{b}_{3}}{\hat{a}_{3}}\right]
		\Rightarrow \frac{\hat{b}_{1}+\hat{b}_{2}}{\hat{a}_{1}+\hat{a}_{2}}<\frac{\hat{b}_{3}}{\hat{a}_{3}},
		\\\label{B}&\left[\frac{\hat{b}_{1}+\hat{b}_{2}}{\hat{a}_{1}+\hat{a}_{2}}<\frac{\hat{b}_{3}}{\hat{a}_{3}} \text{ and } \frac{\hat{b}_{1}}{\hat{a}_{1}}<\frac{\hat{b}_{2}}{\hat{a}_{2}}\right]
		\Rightarrow \frac{\hat{b}_{1}}{\hat{a}_{1}}<\frac{\hat{b}_{2}+\hat{b}_{3}}{\hat{a}_{2}+\hat{a}_{3}},
		\\\label{C and D}&\left[\frac{\hat{b}_{1}}{\hat{a}_{1}}<\frac{\hat{b}_{2}}{\hat{a}_{2}}\text{ and } \frac{\hat{b}_{2}}{\hat{a}_{2}}<\frac{\hat{b}_{3}}{\hat{a}_{3}}\right]\Rightarrow \left[\frac{\hat{b}_{1}}{\hat{a}_{1}}<\frac{\hat{b}_{2}+\hat{b}_{3}}{\hat{a}_{2}+\hat{a}_{3}}\text{ and } \frac{\hat{b}_{1}+\hat{b}_{2}}{\hat{a}_{1}+\hat{a}_{2}}<\frac{\hat{b}_{3}}{\hat{a}_{3}}\right],
		\\\label{E}&\left[\frac{\hat{b}_{1}+\hat{b}_{2}}{\hat{a}_{1}+\hat{a}_{2}}<\frac{\hat{b}_{3}}{\hat{a}_{3}}\text{ and }\left[\frac{\hat{b}_{1}}{\hat{a}_{1}}\geq\frac{\hat{b}_{2}}{\hat{a}_{2}} \text{ or }\frac{\hat{b}_{1}}{\hat{a}_{1}}\geq\frac{\hat{b}_{2}+\hat{b}_{3}}{\hat{a}_{2}+\hat{_{3}}}\right]\right]\Rightarrow \frac{\hat{b}_{2}}{\hat{a}_{2}}<\frac{\hat{b}_{3}}{\hat{a}_{3}},
		\\\label{F}&\left[\frac{\hat{b}_{1}}{\hat{a}_{1}}<\frac{\hat{b}_{2}+\hat{b}_{3}}{\hat{a}_{2}+\hat{a}_{3}}\text{ and }\left[\frac{\hat{b}_{2}}{\hat{a}_{2}}\geq\frac{\hat{b}_{3}}{\hat{a}_{3}}\text{ or }\frac{\hat{b}_{1}+\hat{b}_{2}}{\hat{a}_{1}+\hat{a}_{2}}\geq\frac{\hat{b}_{3}}{\hat{a}_{3}}\right]\right]\Rightarrow\frac{\hat{b}_{1}}{\hat{a}_{1}}<\frac{\hat{b}_{2}}{\hat{a}_{2}},
	\end{align}
	as well as
	\begin{align}
		\label{G}&\left[\frac{\hat{b}_{1}}{\hat{a}_{1}}\geq \frac{\hat{b}_{2}+\hat{b}_{3}}{\hat{a}_{2}+\hat{a}_{3}}\text{ and }\frac{\hat{b}_{2}}{\hat{a}_{2}}\geq\frac{\hat{b}_{3}}{\hat{a}_{3}}\right]\Rightarrow\frac{\hat{b}_{1}+\hat{b}_{2}}{\hat{a}_{1}+\hat{a}_{2}}\geq \frac{\hat{b}_{3}}{\hat{a}_{3}},
		\\\label{H}&\left[\frac{\hat{b}_{1}+\hat{b}_{2}}{\hat{a}_{1}+\hat{a}_{2}}\geq\frac{\hat{b}_{3}}{\hat{a}_{3}}\text{ and }\frac{\hat{b}_{1}}{\hat{a}_{1}}\geq\frac{\hat{b}_{2}}{\hat{a}_{2}}\right]\Rightarrow\frac{\hat{b}_{1}}{\hat{a}_{1}}\geq \frac{\hat{b}_{2}+\hat{b}_{3}}{\hat{a}_{2}+\hat{a}_{3}}, 
		\\\label{I and J} &\left[\frac{\hat{b}_{1}}{\hat{a}_{1}}\geq\frac{\hat{b}_{2}}{\hat{a}_{2}}\text{ and }\frac{\hat{b}_{2}}{\hat{a}_{2}}\geq\frac{\hat{b}_{3}}{\hat{a}_{3}}\right]\Rightarrow\left[\frac{\hat{b}_{1}+\hat{b}_{2}}{\hat{a}_{1}+\hat{a}_{2}}\geq \frac{\hat{b}_{3}}{\hat{a}_{3}}\text{ and }\frac{\hat{b}_{1}}{\hat{a}_{1}}\geq \frac{\hat{b}_{2}+\hat{b}_{3}}{\hat{a}_{2}+\hat{a}_{3}} \right],
		\\\label{K}&\left[\frac{\hat{b}_{1}}{\hat{a}_{1}}\geq \frac{\hat{b}_{2}+\hat{b}_{3}}{\hat{a}_{2}+\hat{a}_{3}} \text{ and }\left[\frac{\hat{b}_{2}}{\hat{a}_{2}}<\frac{\hat{b}_{3}}{\hat{a}_{3}}\text{ or } \frac{\hat{b}_{1}+\hat{b}_{2}}{\hat{a}_{1}+\hat{a}_{2}}< \frac{\hat{b}_{3}}{\hat{a}_{3}}\right]\right]\Rightarrow\frac{\hat{b}_{1}}{\hat{a}_{1}}\geq \frac{\hat{b}_{2}}{\hat{a}_{2}}, 
		\\\label{L}&\left[\frac{\hat{b}_{1}+\hat{b}_{2}}{\hat{a}_{1}+\hat{a}_{2}}\geq \frac{\hat{b}_{3}}{\hat{a}_{3}}\text{ and }\left[\frac{\hat{b}_{1}}{\hat{a}_{1}}< \frac{\hat{b}_{2}}{\hat{a}_{2}}\text{ or } \frac{\hat{b}_{1}}{\hat{a}_{1}}< \frac{\hat{b}_{2}+\hat{b}_{3}}{\hat{a}_{2}+\hat{a}_{3}}\right]\right]\Rightarrow\frac{\hat{b}_{2}}{\hat{a}_{2}}\geq \frac{\hat{b}_{3}}{\hat{a}_{3}}.
	\end{align}
	Looking at equations like \eqref{assumptionaibi_intro}, these implications allow us to "shift" the inequality symbols and to "add" or "remove" terms on one side.
	\begin{proof}[Proof of Lemma \ref{lemma intersection of all admissible is optimal}]
		Let $0=u_{0}<u_{1}<...<u_{n_{1}}<u_{n_{1}+1}=\ell$ and $0=v_{0}<v_{1}<...<v_{n_{2}}<v_{n_{2}+1}=\ell$ be admissible. First, assume $\{u_{0},u_{1},\!...,u_{n_{1}+1}\} \cap \{v_{0},v_{1},\!...,v_{n_{2}+1}\}=\{0,\ell\}$. We show 
		\begin{equation}\label{from 1 to ell}
			\sum\nolimits_{1}^{m}<\sum\nolimits_{m+1}^{\ell} \text{ for all } m=1,\!...,\ell-1.
		\end{equation}
		We pick w.l.o.g. $j\geq1$ such that $u_{j}<v_{1}<u_{j+1}$ and show 
		\begin{equation}\label{from 1 to uj+1}
			\sum\nolimits_{1}^{m}<\sum\nolimits_{m+1}^{u_{j+1}} \text{ for all } m=1,\!...,u_{j+1}-1.
		\end{equation} 
		Since $0=u_{0}<u_{1}<...<u_{n_{1}}<u_{n_{1}+1}=\ell$ and $0=v_{0}<v_{1}<...<v_{n_{2}}<v_{n_{2}+1}=\ell$ are admissible, we have
		\begin{align}\label{from uj to uj+1}
			\sum\nolimits_{u_{j}+1}^{m}<\sum\nolimits_{m+1}^{u_{j+1}} &\text{ for all } m=u_{j}+1,\!...,u_{j+1}-1,
			\\\label{from 1 to v1}\sum\nolimits_{1}^{m}<\sum\nolimits_{m+1}^{v_{1}} &\text{ for all } m=1,\!...,v_{1}-1.
		\end{align}
		By \eqref{from uj to uj+1} for $m=v_{1}$, \eqref{from 1 to v1} for $m=u_{j}$ and implication \eqref{C and D}, we have \eqref{from 1 to uj+1} for $m=u_{j}$. From this we get \eqref{from 1 to uj+1} for $m=u_{j},\!...,u_{j+1}-1$ by \eqref{from uj to uj+1} and implication \eqref{A}. In particular, \eqref{from 1 to uj+1} holds for $m=v_{1}$, which implies \eqref{from 1 to uj+1} for $m=1,\!...,v_{1}$ by \eqref{from 1 to v1} and implication \eqref{B}. 
		
		Next, we pick $i$ such that $v_{i}<u_{j+1}<v_{i+1}$, if it exists, and get analogously
		\begin{equation}
			\sum\nolimits_{1}^{m}<\sum\nolimits_{m+1}^{v_{i+1}} \text{ for all } m=1, ...,v_{i+1}-1.
		\end{equation}
		Iterating this, we finally have \eqref{from 1 to ell}. If $\{u_{0},u_{1},\!...,u_{n_{1}+1}\} \cap \{v_{0},v_{1},\!...,v_{n_{2}+1}\}=\{0,\ell_{1},\!...,\ell_{\tilde{n}},\ell\}$ with $0<\ell_{1}<...<\ell_{\tilde{n}}<\ell$, one can apply the procedure to each of the landscapes $(a_{i},b_{i})_{i=1}^{\ell_{1}}$, $(a_{i},b_{i})_{i=\ell_{1}+1}^{\ell_{2}}$,...,$(a_{i},b_{i})_{i=\ell_{\tilde{n}}+1}^{\ell}$.
	\end{proof}
	\begin{proof}[Proof of Lemma \ref{lemma intersection is better}]
		Since $D^{c}_{u}$ and $D^{c}_{v}$ are subsets of $D^{c}_{w}$, we almost have \eqref{equation intersection better v}, but possibly with equality.
		
		By definition of $D^{c}$, finding the optimal $(c_{1},\!...,c_{\ell-1})\in D^{c}_{w}$ is equivalent to finding all optimal times for a BBM among obstacles $(a_{j},b_{j})_{j={w_{i}+1}}^{w_{i+1}}$. By Lemma \ref{lemma intersection of all admissible is optimal}, the obstacle landscape $(a_{j},b_{j})_{j={w_{i}+1}}^{w_{i+1}}$ satisfies the assumption of late expensive obstacles. Hence, by Proposition \ref{proposition optimal times in case of late expensive obstacles}, the unique optimal $(c_{1},\!...,c_{\ell-1})\in D^{c}_{w}$ has to satisfy $c_{m}>0$ for all $m=w_{i}+1,\!...,w_{i+1}-1$ and $i=0,\!...,n_{3}$. By existence of $\hat{i}$ and $\hat{j}$, we get the strict inequality \eqref{equation intersection better v}.
	\end{proof}
	\begin{proof}[Proof of Lemma \ref{lemma optimal is a subset of indices for which assumption fails}]
		We have to show
		\begin{equation}\label{from (si)+1 to si+1}
			\sum\nolimits_{s_{i}+1}^{m}<\sum\nolimits_{m+1}^{s_{i+1}}\text{ for all }i=0,\!...,n\text{ and }m=s_{i}+1,\!...,s_{i+1}-1.
		\end{equation}
		By definition of $\{s_{1},\!...,s_{n}\}$, we have
		\begin{equation}\label{definition s1,...,sn}
			\sum\nolimits_{1}^{m}<\sum\nolimits_{m+1}^{\ell} \text{ if and only if } m\notin \{s_{1},\!...,s_{n}\}.
		\end{equation}
		Now \eqref{definition s1,...,sn} and implication \eqref{E} respectively \eqref{F} imply
		\begin{align}
			\label{first cut (si)+1 left}\sum\nolimits_{s_{i}+1}^{m}<\sum\nolimits_{m+1}^{\ell}&\text{ for all }m=s_{i}+1,\!...,s_{i+1}-1,
			\\\label{first cut right si+1}\sum\nolimits_{1}^{m}<\sum\nolimits_{m+1}^{s_{i+1}} &\text{ for all } m=s_{i}+1,\!...,s_{i+1}-1.
		\end{align}
		Using \eqref{first cut right si+1}, \eqref{definition s1,...,sn} and implication \eqref{L}, we get
		\begin{equation}\label{geq bei (si)+1 to si+1 to ell}
			\sum\nolimits_{m+1}^{s_{i+1}}\geq\sum\nolimits_{s_{i+1}+1}^{\ell} \text{ for all }m=s_{i}+1,\!...,s_{i+1}-1,
		\end{equation}
		which leads, in combination with \eqref{first cut (si)+1 left} and implication \eqref{F}, to \eqref{from (si)+1 to si+1}.
	\end{proof}
	\begin{proof}[Proof of Lemma \ref{lemma sufficent to check inequality for si}]
		Assume the statement was false. We show that this contradicts Lemma \ref{lemma optimal is a subset of indices for which assumption fails}. 
		
		 Then there exists  some $i\in\{\tilde{i}+1,\!...,\tilde{j}-1\}$ and $m\in\{s_{i}+1,\!...,s_{i+1}-1\}$ such that
		\begin{equation}\label{assume < does not hold for all indices between them}
			\sum\nolimits_{s_{\tilde{i}}+1}^{m}\geq\sum\nolimits_{m+1}^{s_{\tilde{j}}}.
		\end{equation}
		By \eqref{< for indices for which assumption fails} and implication \eqref{K} respectively \eqref{L}, this would imply
		\begin{align}
			&\label{cut right si+1}
			\sum\nolimits_{s_{\tilde{i}}+1}^{m}\geq\sum\nolimits_{m+1}^{s_{i+1}},\quad \mbox{and}\quad
			\sum\nolimits_{s_{i}+1}^{m}\geq\sum\nolimits_{m+1}^{s_{\tilde{j}}}.
		\end{align}
		Now \eqref{cut right si+1}, \eqref{< for indices for which assumption fails} and implication \eqref{F} would lead to
		\begin{equation}
			\sum\nolimits_{s_{\tilde{i}}+1}^{s_{i}}<\sum\nolimits_{s_{i}+1}^{m}.
		\end{equation}
	Together with with \eqref{cut right si+1} and implication \eqref{L}, we get
		\begin{equation}
			\sum\nolimits_{s_{i}+1}^{m}\geq\sum\nolimits_{m+1}^{s_{i+1}}.
		\end{equation}
		This his is a contradiction to the admissibility of $0=s_{0}<s_{1}<...<s_{n}<s_{n+1}=\ell$ and the claim follows.
	\end{proof}
	The example in Figure \ref{fig:picture-obstacles-block-division} illustrates the idea of dividing the obstacle landscape into blocks.
	\begin{figure}[H]
		\centering
		\includegraphics[width=1\linewidth]{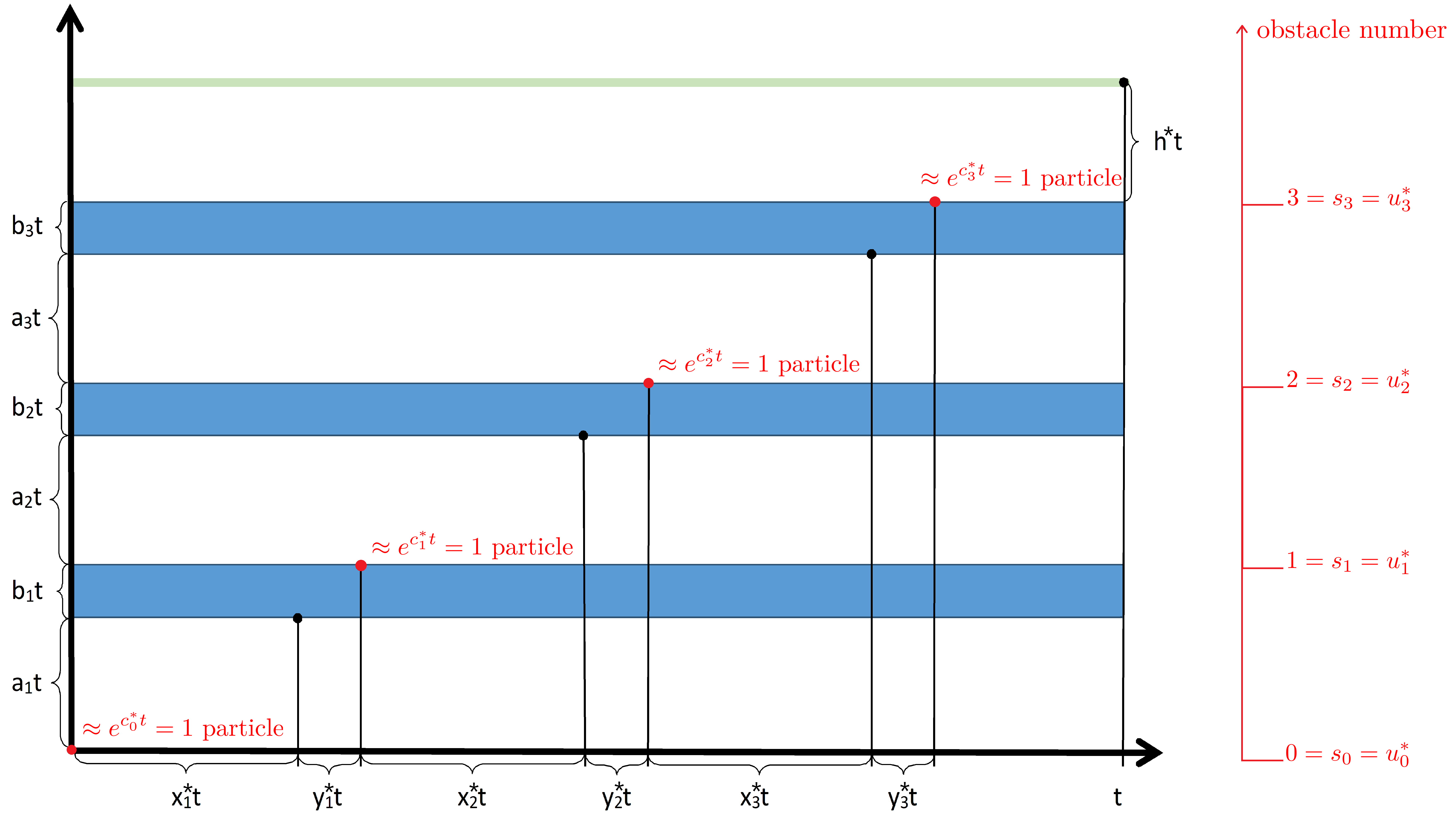}
		\caption{In our example, we have $a_{1}=a_{2}=a_{3}$ and $b_{1}=b_{2}=b_{3}$. As illustrated on the red axis, this implies $\{s_{1},\!...,s_{n}\}=\{1,2\}$ and also $0=u_{0}^{*}<1=u_{1}^{*}<2=u_{2}^{*}<3=u_{3}^{*}$. I.e. each obstacle and the corresponding branching area build a separate block. Heuristically, the optimal strategy has $\approx 1$ particle at the red dots.}
		\label{fig:picture-obstacles-block-division}
	\end{figure}
\subsection{Maximal particle as a descendant of one of the first particles above all obstacles}\label{First particle above obstacles will be maximal}
	In this subsection, we show that the argmax of \eqref{optimize} over $\hat{D}$ equals the argmax of \eqref{optimize} over $D$. We start with a lemma relating the two domains to each other.
	\begin{lem}\label{lemma D not empty iff Dhat not empty}
		The domain $\hat{D}$ is not empty if and only if $D$ is not empty. Furthermore, we have	
		\begin{equation}\label{xi*yi* fastest strategy}
			\sum_{i=1}^{\ell}(x_{i}^{*}+y_{i}^{*})\leq \sum_{i=1}^{\ell}(x_{i}+y_{i}),
		\end{equation}
		for all $(x_{1},y_{1},\!...,x_{\ell},y_{\ell})\in D$. 
	\end{lem}
	\begin{proof}
		$\hat{D}$ is a subset of $D$ by definition. Conversely, for all $(x_{1},y_{1},\!...,x_{\ell},y_{\ell})\in D$ we have that
		\begin{equation}\label{modify yell to be in Dhat}
			\left(x_{1},y_{1},\!...,x_{\ell},\frac{b_{\ell}^{2}}{2\left(\sum_{i=1}^{\ell-1}\left(x_{i}-\frac{a_{i}^{2}}{2x_{i}}-\frac{b_{i}^{2}}{2y_{i}}\right)+x_{\ell}-\frac{a_{\ell}^{2}}{2x_{\ell}}\right)}\right) \in \hat{D}.
		\end{equation}
		Since the last entry of \eqref{modify yell to be in Dhat} is not greater than $y_{\ell}$, \eqref{xi*yi* fastest strategy} follows by Theorem \ref{theorem candidate in the general case Dhat Variante}. 
	\end{proof}
	Now, we state the main result of this section. Our candidate for the first order of the maximum of BBM among obstacles is $\sum_{i=1}^{\ell}(a_{i}+b_{i})t+h^{*}t$ where $h^{*}$ is equal to \eqref{h* in the general case}.
	\begin{theorem}\label{theorem candidate in the general case}
		If $D$ is not empty, the maximum of \eqref{optimize} over $D$ equals $(h^{*})^{2}/2$ with
		\begin{align}\label{h* in the general case}
			h^{*}=\sqrt{2}\left(1-\sum_{i=0}^{n^{*}}\left(\sum_{j=u_{i}^{*}+1}^{u_{i+1}^{*}}a_{j}f(\tilde{c}_{i})+\frac{\left(\sum_{j=u_{i}^{*}+1}^{u_{i+1}^{*}}b_{j}\right)^{2}}{2\sum_{j=u_{i}^{*}+1}^{u_{i+1}^{*}}a_{j}\left(f(\tilde{c}_{i})-\frac{1}{2f(\tilde{c}_{i})}\right)}\right)\right).
		\end{align}
	\end{theorem}

	\begin{proof}
		Assume w.l.o.g. $\sum_{i=1}^{\ell}\left(x_{i}^{*}+y_{i}^{*}\right)<1$. (Otherwise $D$ and $\hat{D}$ consist of only one element.) The claim follows by Theorem \ref{theorem candidate in the general case Dhat Variante} if we show
		\begin{equation}\label{compare boundary and interior}
			\left(1-\sum_{i=1}^{\ell}(x_{i}^{*}+y_{i}^{*})\right)^{2}>
			\left(1-\sum_{i=1}^{\ell}(x_{i}+y_{i})\right)\left(1-\sum_{i=1}^{\ell}\left( y_{i}+\frac{a_{i}^{2}}{2x_{i}}+\frac{b_{i}^{2}}{2y_{i}}\right)\right)
		\end{equation}
		for all $(x_{1},y_{1},\!...,x_{\ell},y_{\ell})\in D$ that do not satisfy \eqref{domain1} with equality for $m=\ell$. If some $(x_{1},y_{1},\!...,x_{\ell},y_{\ell})$ in the interior of $D$ maximized \eqref{optimize} over $D$, it would be a critical point. In particular, the derivative of \eqref{optimize} with respect to $x_{1}$ would have to satisfy
		\begin{equation}\label{assume interior->BEO x1}
			-\left(1-\sum_{i=1}^{\ell}\left( y_{i}+\frac{a_{i}^{2}}{2x_{i}}+\frac{b_{i}^{2}}{2y_{i}}\right)\right)+\left(1-\sum_{i=1}^{\ell}(x_{i}+y_{i})\right)\dfrac{a_{1}^{2}}{2x_{1}^{2}}=0.
		\end{equation}
		Plugging \eqref{assume interior->BEO x1} into \eqref{compare boundary and interior}, it remains to show
		\begin{equation}\label{compare boundary and interior 2}
			\left(1-\sum_{i=1}^{\ell}(x_{i}^{*}+y_{i}^{*})\right)^{2}>\left(1-\sum_{i=1}^{\ell}(x_{i}+y_{i})\right)^{2}\dfrac{a_{1}^{2}}{2x_{1}^{2}}.
		\end{equation}
		By \eqref{domain1} and $y_{1}>0$, we have $x_{1}> a_{1}/\sqrt{2}$, which implies $a_{1}^{2}/(2x_{1}^{2})<1$. Furthermore, we have $1-\sum_{i=1}^{\ell}(x_{i}^{*}+y_{i}^{*})\leq1-\sum_{i=1}^{\ell}(x_{i}+y_{i})$ by \eqref{xi*yi* fastest strategy}. Consequently, \eqref{compare boundary and interior 2} is indeed true.
		
		It remains to look at the boundary of $D$. $x_{i}\to 0$ or $y_{i}\to 0$ would contradict \eqref{domain1} because $x_{1}+...+x_{\ell}\leq1$ is not able to compensate $a_{i}^{2}/(2x_{i})\to\infty$ or $b_{i}^{2}/(2y_{i})\to\infty$. Equality in \eqref{domain3} would imply that the r.h.s of \eqref{compare boundary and interior} equals $0$ whereas the l.h.s is assumed to be strictly positive. If \eqref{domain1} holds with equality for some $\tilde{m}<\ell$, we have to show
		\begin{multline}\label{compare boundary and other boundary}
			\left(1-\sum_{i=1}^{\ell}(x_{i}^{*}+y_{i}^{*})\right)^{2}>\\
			\left(1-\sum_{i=1}^{\ell}(x_{i}+y_{i})\right)\left(\sum_{i=\tilde{m}+1}^{\ell}\left( x_{i}-\frac{a_{i}^{2}}{2x_{i}}-\frac{b_{i}^{2}}{2y_{i}}\right)+ 1-\sum_{i=1}^{\ell}(x_{i}+y_{i})\right).
		\end{multline}
		If one wants to maximize the r.h.s. of \eqref{compare boundary and other boundary} for given $(x_{1},y_{1},\!...,x_{\tilde{m}},y_{\tilde{m}})$, one can argue as above: If an interior solution was optimal, it would have to be a critical point and satisfy the first order condition with respect to $x_{\tilde{m}+1}$, which is given by
		\begin{equation}\label{assume interior->BEO xm+1}
			\left(\sum_{i=\tilde{m}+1}^{\ell}\left( x_{i}-\frac{a_{i}^{2}}{2x_{i}}-\frac{b_{i}^{2}}{2y_{i}}\right)+ 1-\sum_{i=1}^{\ell}(x_{i}+y_{i})\right)=\left(1-\sum_{i=1}^{\ell}(x_{i}+y_{i})\right)\frac{a_{\tilde{m}+1}^{2}}{2x_{\tilde{m}+1}^{2}}.
		\end{equation}
		Plugging \eqref{assume interior->BEO xm+1} into \eqref{compare boundary and other boundary}, we have to show
		\begin{equation}
			\left(1-\sum_{i=1}^{\ell}(x_{i}^{*}+y_{i}^{*})\right)^{2}>
			\left(1-\sum_{i=1}^{\ell}(x_{i}+y_{i})\right)^{2}\frac{a_{\tilde{m}+1}^{2}}{2x_{\tilde{m}+1}^{2}}.
		\end{equation}
		This is true as $x_{\tilde{m}+1}>a_{\tilde{m}+1}/\sqrt{2}$, by \eqref{domain1} for $m=\tilde{m}+1$ and the choice of $\tilde{m}$, and $1-\sum_{i=1}^{\ell}(x_{i}^{*}+y_{i}^{*})\geq1- \sum_{i=1}^{\ell}(x_{i}+y_{i})$ by \eqref{xi*yi* fastest strategy}. Hence, to show \eqref{compare boundary and interior}, we only have to consider those elements with equality in \eqref{domain1} for some $m>\tilde{m}$. Iterating this argument, one gets \eqref{compare boundary and interior} for all elements in $D$ that do not satisfy \eqref{domain1} with equality for $m=\ell$.
	\end{proof}
\section{Proof of Theorem \ref{theorem main result}}\label{section almost sure}
\subsection{Upper bound}\label{subsection upper bound}
	 We prove that, as $t \to \infty$, there exists almost surely no particle above $\sum_{i=1}^{\ell}(a_{i}+b_{i})t+(h^{*}+\epsilon)t$ (indicated by the horizontal red line in  Figure \ref{fig:picture-upper-bound-with-red-intervals}) at time $t$. 
	
	\begin{prop}\label{proposition upper bound}
		Let $(a_{i},b_{i})_{i=1}^{\ell}$ be some obstacle landscape such that $\sum_{i=1}^{\ell}(x_{i}^{*}+y_{i}^{*})\leq1$. Then, for any $\epsilon>0$, we have, almost surely,
		\begin{equation}\label{upper bound almost sure convergence}
			\lim\limits_{t\to\infty}\frac{\max_{k\leq n(t)}X_{k}(t)}{t}< \sum_{i=1}^{\ell}(a_{i}+b_{i})+h^{*}+\epsilon.
		\end{equation}
	\end{prop}
	\begin{proof}
		We show that for all $\epsilon>0$, there exists some constant $C_{2}>0$ such that
		\begin{equation}\label{upper bound convergence in probability}
			\mathbb{P}\left(\exists k\leq n(t) : X_{k}(t)\geq\left(\sum_{i=1}^{\ell}(a_{i}+b_{i})+h^{*}+\epsilon\right)t\right)\lesssim e^{-C_{2}t}.
		\end{equation}
		Since the r.h.s of \eqref{upper bound convergence in probability} is integrable with respect to $t$, \eqref{upper bound almost sure convergence} follows by the Borel-Cantelli Lemma and approximation arguments (see e.g. \cite{ABK_G}). We define, for $m=1,\!...,\ell$ and $k=1,\!...,n(t)$,
		\begin{align}
			\tau_{2m-1}^{k}&=\sup\left\{s\leq t : X_{k}(s)\leq \sum_{i=1}^{m-1}(a_{i}+b_{i})t+a_{m}t\right\},
			\\\tau_{2m}^{k}&=\inf\left\{s\geq\tau_{2m-1}^{k} : X_{k}(s)\geq \sum_{i=1}^{m}(a_{i}+b_{i})t\right\},
			\\\mathcal{X}(k)&=\left\{X_{k}(t)\geq\left(\sum_{i=1}^{\ell}(a_{i}+b_{i})+h^{*}+\epsilon\right)t\right\},
			\\\mathcal{X}(k,n_{1},\!...,n_{2\ell})&=
			\mathcal{X}(k)\cap \left\{\tau_{i}^{k}\in [n_{i}-1,n_{i}] \text{ for all }i=1,\!...,2\ell\right\}.
		\end{align}
		The events $\mathcal{X}(k)$ and $ \mathcal{X}(k,n_1,\dots,n_{2l})$ are visualised in Figure \ref{fig:picture-upper-bound-with-red-intervals}.  
		\begin{figure}[H]
		\centering
		\includegraphics[width=1\linewidth]{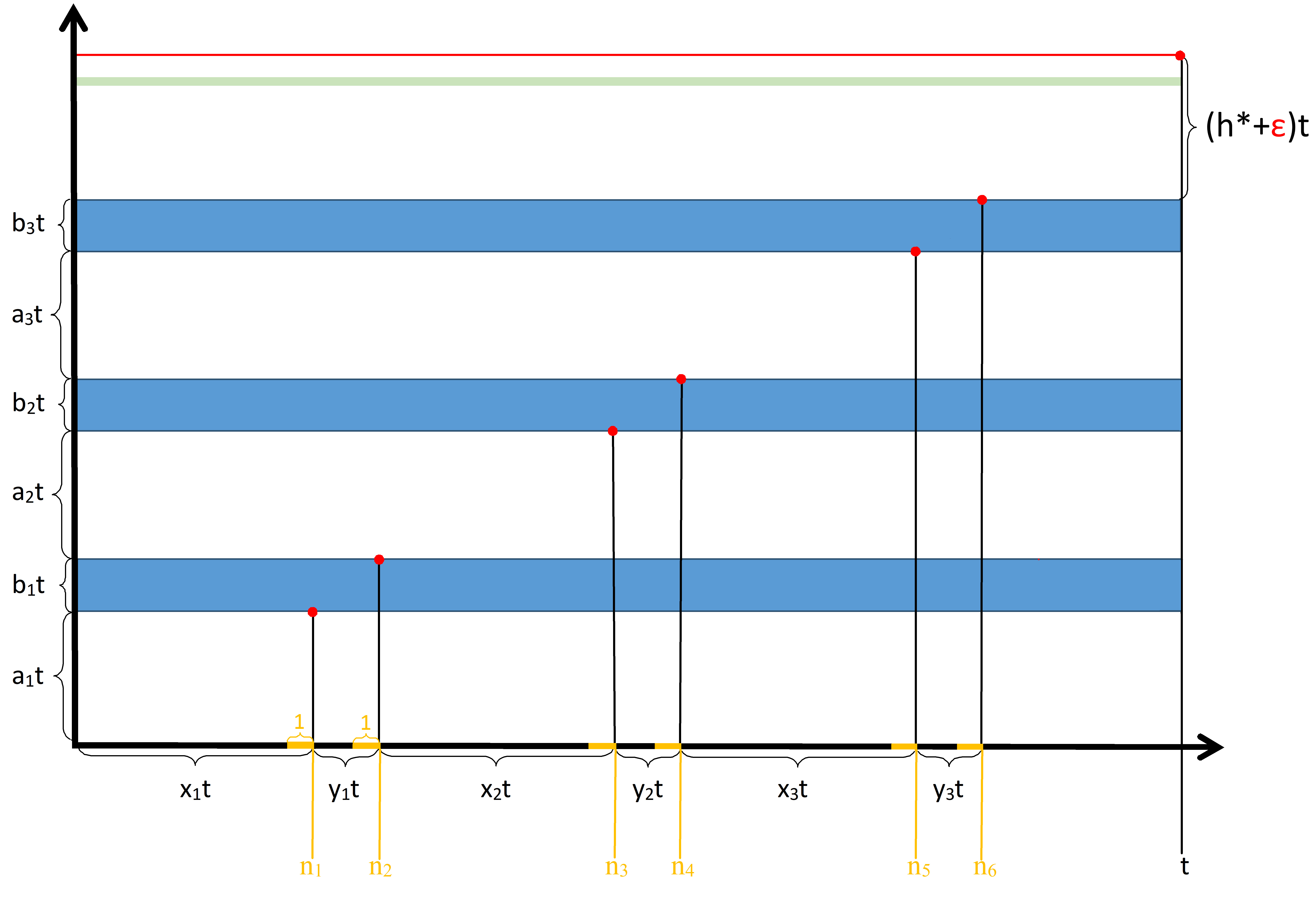}
		\caption{The horizontal red line is at height $\sum_{i=1}^{\ell}(a_{i}+b_{i})t+(h^{*}+\epsilon)t$. A particle follows the strategy marked with red dots if its last time below the $m$-th obstacle is in the interval $[n_{2m-1}-1,n_{2m-1}]$ and its next time above the $m$-th obstacle is in the interval $[n_{2m}-1,n_{2m}]$. At time $t$ it is above the horizontal red line.
		}
		\label{fig:picture-upper-bound-with-red-intervals}
	\end{figure}	
		The endpoints $(n_{1},\!...,n_{2\ell})$ of the intervals are in the domain 
		\begin{equation}
			D_{1}=\left\{(n_{1},\!...,n_{2\ell})\in([1,t]\cap \mathbb{N})^{2\ell} : n_{1}\leq...\leq n_{2\ell}\right\}.
		\end{equation}
	We rewrite the probability in \eqref{upper bound convergence in probability} as
		\begin{align}
		\mathbb{P}\left(\sum_{k=1}^{n(t)}\mathbbm{1}_{\mathcal{X}(k)}\geq 1\right)&
			\leq\sum_{(n_{1},\!...,n_{2\ell})\in D_{1}}\mathbb{P}\left(\sum_{k=1}^{n(t)}\mathbbm{1}_{\mathcal{X}(k,n_{1},\!...,n_{2\ell})}\geq 1\right)\label{union bound for upper bound},
		\end{align}
by a union bound.		From now on, we set
		\begin{equation}\label{relation new and old notation}
			(x_{1},y_{1},\!...,x_{\ell},y_{\ell})=\left(\frac{n_{1}}{t}, \frac{n_{2}-n_{1}}{t},\!...,\frac{n_{2\ell-1}-n_{2\ell-2}}{t},\frac{n_{2\ell}-n_{2\ell-1}}{t}\right).
		\end{equation}
		Furthermore, we define
		\begin{align}
			D_{2}=&\Big\{(n_{1},\!...,n_{2\ell})\in D_{1} : (x_{1},y_{1},\!...,x_{\ell},y_{\ell}) \in D\Big\},
			\\D_{3}=&\Big\{(n_{1},\!...,n_{2\ell})\in D_{1} : (x_{1},y_{1},\!...,x_{\ell},y_{\ell}) \text{ does not satisfy }\eqref{domain4}\Big\},
			\\D_{4}=&\Big\{(n_{1},\!...,n_{2\ell})\in D_{1} : (x_{1},y_{1},\!...,x_{\ell},y_{\ell}) \text{ satisfies }\eqref{domain4}\text{ but not }\eqref{domain1}\Big\}.
		\end{align}
		The domain $D_{2}$ considers those elements of $D_{1}$ that correspond to the domain $D$ of the optimization problem in Section \ref{subsection optimization problem}. The other domains are related to violating a condition of $D$. (Condition \eqref{domain3} can not be violated because of $n_{1}\leq...\leq n_{2\ell}$.) As $D_{1}=D_{2}\cup D_{3}\cup D_{4}$, \eqref{upper bound convergence in probability} follows, once we have proven the following three lemmas.
		\begin{lem}\label{lemma D2}
		Under the assumption of Proposition \ref{upper bound cover 2}, there is some constant $C_{3}>0$, depending on $\epsilon$, such that 
		\begin{equation}\label{upper bound cover 2}
			\sum_{(n_{1},\!...,n_{2\ell})\in D_{2}}\mathbb{P}\left(\sum_{k=1}^{n(t)}\mathbbm{1}_{\mathcal{X}(k,n_{1},\!...,n_{2\ell})}\geq 1\right)\lesssim e^{-C_{3}t}.
		\end{equation}
	\end{lem}
	\begin{lem}\label{lemma D3}
		Under the assumption of Proposition \ref{upper bound cover 2}, there is some constant $C_{4}>0$ such that
		\begin{equation}\label{eqD3}
			\sum_{(n_{1},\!...,n_{2\ell})\in D_{3}}\mathbb{P}\left(\sum_{k=1}^{n(t)}\mathbbm{1}_{\mathcal{X}(k,n_{1},\!...,n_{2\ell})}\geq 1\right)\lesssim e^{-C_{4}t}.
		\end{equation}
	\end{lem}
	\begin{lem}\label{lemma D5}
		Under the assumption of Proposition \ref{upper bound cover 2}, there is some constant $C_{5}>0$, depending on $\epsilon$, such that
		\begin{equation}\label{convergence D5}
		\sum_{(n_{1},\!...,n_{2\ell})\in D_{4}}\mathbb{P}\left(\sum_{k=1}^{n(t)}\mathbbm{1}_{\mathcal{X}(k,n_{1},\!...,n_{2\ell})}\geq 1\right)\lesssim e^{-C_{5}t}.
		\end{equation}
	\end{lem}
	\end{proof}
	
	\begin{proof}[Proof of Lemma \ref{lemma D2}]
		By Markov's inequality, we have
		\begin{equation}\label{Markov applied to one summand of union bound for upper bound}
			\sum_{(n_{1},\!...,n_{2\ell})\in D_{2}}\mathbb{P}\left(\sum_{k=1}^{n(t)}\mathbbm{1}_{\mathcal{X}(k,n_{1},\!...,n_{2\ell})}\geq 1\right)\leq t^{2\ell}\max\limits_{(n_{1},\!...,n_{2\ell})\in D_{2}}\mathbb{E}\left[\sum_{k=1}^{n(t)}\mathbbm{1}_{\mathcal{X}(k,n_{1},\!...,n_{2\ell})}\right].
		\end{equation} 
		Next, we bound the expectation on the r.h.s. in   \eqref{Markov applied to one summand of union bound for upper bound} from above. Comparing  BBM among obstacles with standard BBM and using \eqref{expectation branching areas}, we get
		\begin{equation}\label{expectation upper bound first dot}
			\mathbb{E}\left[\sum_{k=1}^{n(x_{1}t)}\mathbbm{1}_{\left\{\tau_{1}^{k}\in[n_{1}-1,n_{1}]\right\}}\right]\lesssim\exp\left(x_{1}t-\frac{a_{1}^{2}t}{2x_{1}}\right).
		\end{equation}
		Next, note that the contribution of particles with location above  $a_{1}t+t^{(3/4)}$ to the expectation in   \eqref{expectation upper bound first dot} is negligible compared to the r.h.s of  \eqref{expectation upper bound first dot}, as by Lemma \ref{lemma for obstacles}
		\begin{equation}
			\mathbb{E}\left[\sum_{k=1}^{n(x_{1}t)}\mathbbm{1}_{\left\{\tau_{1}^{k}\in[n_{1}-1,n_{1}]\right\}\cap\left\{X_{k}(x_{1}t)>a_{1}t+t^{\frac{3}{4}}\right\}}\right]\lesssim\exp\left(x_{1}t-\frac{a_{1}^{2}t}{2x_{1}}-\frac{t^{\frac{3}{2}}}{2}\right).
		\end{equation}
		Hence, we bound the way to the next branching area from below by $b_{1}t-t^{(3/4)}$ and the available time from above by $y_{1}t$. Between $\tau_{1}^{k}$ and $\tau_{2}^{k}$, $X_{k}$ does not branch. Possible branching in the small interval $[\tau_{2}^{k},n_{2}]$ can be taken into the error term. Then we get, by Lemma \ref{lemma for obstacles},
		\begin{equation}
			\mathbb{E}\left[\sum_{k=1}^{n\left((x_{1}+y_{1})t\right)}\mathbbm{1}_{\left\{\tau_{1}^{k}\in[n_{1}-1,n_{1}]\right\}\cap\left\{\tau_{2}^{k}\in[n_{2}-1,n_{2}]\right\}}\right]\lesssim\exp\left(x_{1}t-\frac{a_{1}^{2}t}{2x_{1}}-\frac{b_{1}^{2}t}{2y_{1}}\right).
		\end{equation}
		Again, only particles in $[(a_{1}+b_{1})t,(a_{1}+b_{1})t+t^{(3/4)}]$ are relevant. Iterating this procedure, we obtain the upper bound
		\begin{align}\label{expectation upper bound all dots}
			&\mathbb{E}\left[\sum_{k=1}^{n(t)}\mathbbm{1}_{\mathcal{X}(k,n_{1},\!...,n_{2\ell})}\right]\nonumber\\
		&	\lesssim \exp\Bigg(\Bigg[\sum_{i=1}^{\ell}\left( x_{i}-\frac{a_{i}^{2}}{2x_{i}}-\frac{b_{i}^{2}}{2y_{i}}\right)
			+ \left(1-\sum_{i=1}^{\ell}(x_{i}+y_{i})\right)-\frac{\left(h^{*}+\epsilon\right)^{2}}{2\left(1-\sum_{i=1}^{\ell}(x_{i}+y_{i})\right)}\Bigg]t\Bigg)\nonumber \\
			&\equiv M\left(x_{1},y_{1},\!...,x_{\ell},y_{\ell}\right).
			\end{align}
		Hence, the r.h.s. of \eqref{Markov applied to one summand of union bound for upper bound} is bounded from above by
		\begin{align}
		\label{optimization problem appears}&\lesssim \max\limits_{(x_{1},y_{1},\!...,x_{\ell},y_{\ell})\in D}M\left(x_{1},y_{1},\!...,x_{\ell},y_{\ell}\right).
		\end{align}
		By Theorem \ref{theorem candidate in the general case} and simple algebraic manipulations, \eqref{optimization problem appears} can be bounded from above by
		\begin{equation}
			\label{exponentially fast D2}
			\exp\left(\left[\sum_{i=1}^{\ell}\left(x_{i}^{*} -\frac{a_{i}^{2}}{2x_{i}^{*}}-\frac{b_{i}^{2}}{2y_{i}^{*}}\right)
			+1-\sum_{i=1}^{\ell}(x_{i}^{*}+y_{i}^{*})-\frac{(h^{*}+\epsilon)^{2}}{2\left(1-\sum_{i=1}^{\ell}(x_{i}^{*}+y_{i}^{*})\right)}\right]t\right).
		\end{equation}
		The exponent in \eqref{exponentially fast D2} is strictly negative, because
		\begin{equation}
			2\left(1-\sum_{i=1}^{\ell}(x_{i}^{*}+y_{i}^{*})\right)\left(\sum_{i=1}^{\ell}\left(x_{i}^{*} -\frac{a_{i}^{2}}{2x_{i}^{*}}-\frac{b_{i}^{2}}{2y_{i}^{*}}\right)
			+1-\sum_{i=1}^{\ell}(x_{i}^{*}+y_{i}^{*})\right)
		\end{equation}
		equals $(h^{*})^{2}$ which is strictly smaller than $(h^{*}+\epsilon)^{2}$.
	\end{proof}
		\begin{proof}[Proof of Lemma \ref{lemma D3}]
		If $(n_{1},\!...,n_{2\ell}) \in D_{3}$, we have $n_{i}=n_{i+1}$ for some $i\in\{1,\!...,2\ell\}$. In particular, a whole obstacle or branching area has to be crossed during the time interval $[n_{i}-1,n_{i}]$. For large $t$, the size of this obstacle respectively branching area is bounded from below by $t^{3/4}$. The expected number of particles at time $n_i$ is bounded from above by  $e^{t}$.  
		Hence, by Markov's inequality and Lemma \ref{lemma for obstacles}, we  bound the l.h.s. in \eqref{eqD3} from above by
		\begin{align}
		 t^{2\ell}\max\limits_{(n_{1},\!...,n_{2\ell})\in D_{3}}\mathbb{P}\left(\sum_{k=1}^{n(t)}\mathbbm{1}_{\mathcal{X}(k,n_{1},\!...,n_{2\ell})}\geq 1\right)
			&\lesssim\exp\left(t-\frac{t^{\frac{3}{2}}}{2}\right),
		\end{align}
		which is smaller than $e^{-t}$ for large enough $t$.
	\end{proof}
			\begin{proof}[Proof of Lemma \ref{lemma D5}]
			For $(n_1,\dots,n_{2l})$ given, let $m_1$ be the first index such that \eqref{domain1} does not hold. Moreover, we define $m_j$ through
			\begin{align}
			&\label{ok until mj}\sum_{i=m_{j-1}+1}^{m_j}\left(x_i-\frac{a_i^2}{2x_i}-\frac{b_i^2}{2y_i}\right)<0\quad\text{and}\nonumber\\
			&\sum_{i=m_{j-1}+1}^m\left(x_i-\frac{a_i^2}{2x_i}-\frac{b_i^2}{2y_i}\right)\geq0\text{ for all }m=m_{j-1}+1,\!...,m_j-1.
			\end{align}
	For notational convenience, we keep the dependence of $m_j$ on $(n_1,\dots,n_{2l})$ implicit.	
Next, we define 
$D_4(\delta,+)$ and $D_4(\delta,-)$ through
\begin{align}
D_4(\delta,+)&=\left\{(n_1,\dots,n_{2l})\in D_4: \exists m_j : \sum_{i=m_{j-1}+1}^{m_j}\left(x_i-\frac{a_i^2}{2x_i}-\frac{b_i^2}{2y_i}\right)<-\delta 
\right\},
\nonumber\\
D_4(\delta,-)&= D_4\setminus D_4(\delta,+).
\end{align}
		For $(n_1,\dots,n_{2l})\in D_4(\delta,+)$ we have, by Markov's inequality,
			\begin{equation}\label{bound for at least deltahat to small}
			\sum_{(n_{1},\!...,n_{2\ell})\in D_4(\delta,+)}\mathbb{P}\left(\sum_{k=1}^{n(t)}\mathbbm{1}_{\mathcal{X}(k,n_{1},\!...,n_{2\ell})}\geq 1\right)\lesssim \exp\left(-\delta t\right).
			\end{equation}
 Let $\overline{D_4(\delta,-)}$ be the closure of $D_4(\delta,-)$
			and let $(\overline{x}_{1},\overline{y}_{1},\!...,\overline{x}_{\ell},\overline{y}_{\ell})$ maximize
			\begin{equation}\label{optimize again}
			\sum_{i=1}^{\ell}\left(x_{i}-\frac{a_{i}^{2}}{2x_{i}}-\frac{b_{i}^{2}}{2y_{i}}\right)
			+1-\sum_{i=1}^{\ell}(x_{i}+y_{i})-\frac{(h^{*}+\epsilon)^{2}}{2\left(1-\sum_{i=1}^{\ell}(x_{i}+y_{i})\right)}
			\end{equation} 
			over $\overline{D_4(\delta,-)}$. Then we have, by Markov's inequality,
			\begin{multline}\label{bound for at most deltahat to small}
			\sum_{(n_{1},\!...,n_{2\ell})\in \overline{D_4(\delta,-)}}\mathbb{P}\left(\sum_{k=1}^{n(t)}\mathbbm{1}_{\mathcal{X}(k,n_{1},\!...,n_{2\ell})}\geq 1\right)\lesssim \\	\exp\Bigg[\Bigg(\sum_{i=1}^{\ell}\left(\overline{x}_{i}-\frac{a_{i}^{2}}{2\overline{x}_{i}}-\frac{b_{i}^{2}}{2\overline{y}_{i}}\right)
			+1-\sum_{i=1}^{\ell}(\overline{x}_{i}+\overline{y}_{i})-\frac{(h^{*}+\epsilon)^{2}}{2\left(1-\sum_{i=1}^{\ell}(\overline{x}_{i}+\overline{y}_{i})\right)}\Bigg)t\Bigg].
			\end{multline}
			By Theorem \ref{theorem candidate in the general case} and simple computations, the maximum of \eqref{optimize again} over $D$ is attained at $(x_1^*,y_1^*,\!...,x_{\ell}^*,y_{\ell}^*)$ and strictly negative. Since $\lim\limits_{\delta\searrow 0}\overline{D_4(\delta,-)}\left(\hat{\delta}\right)\subset D_2$, we can choose $\delta>0$ so small that the exponent on the r.h.s. of \eqref{bound for at most deltahat to small} is also strictly negative. Combining \eqref{bound for at least deltahat to small} and \eqref{bound for at most deltahat to small}, we get \eqref{convergence D5} via a union bound, by choosing $\delta>0$ small enough. 
	\end{proof}
\subsection{Lower bound}\label{subsection lower bound}
In this subsection, we prove the following proposition.

	\begin{prop}\label{proposition lower bound}
		Let $(a_{i},b_{i})_{i=1}^{\ell}$ be some obstacle landscape such that $\sum_{i=1}^{\ell}(x_{i}^{*}+y_{i}^{*})\leq1$. Then, for any $\epsilon>0$, we have, almost surely,
		\begin{equation}\label{lower bound almost sure convergence}
			\lim\limits_{t\to\infty}\frac{\max_{k\leq n(t)}X_{k}(t)}{t}> \sum_{i=1}^{\ell}(a_{i}+b_{i})+h^{*}-\epsilon.
		\end{equation}
	\end{prop}

Before proving Proposition \ref{proposition lower bound}, we need to introduce some notation. For $\delta>0$ and $C>0$, we define, for $m=1,\!...,\ell$, the intervals
\begin{align}\label{defI}
	I_{m}^{A}&=\sum_{i=1}^{m-1}(a_{i}+b_{i})t+\left[ (a_{m}-\delta)t, a_{m}t\right],\quad
I_{m}^{B}=\sum_{i=1}^{m}(a_{i}+b_{i})t+\left[\delta t,\delta t+C\right].
\end{align}
Moreover, we define, for $k\leq n(t)$, the events
\begin{align}
A_{m}^{k}&=\left\{X_{k}\left(\sum_{i=1}^{m-1}(x_{i}^{*}+y_{i}^{*})t+x_{m}^{*}t+\hat{\delta}t\right)\in I_{m}^{A}\right\},\\
B_{m}^{k}&=\left\{X_{k}\left(\sum_{i=1}^{m}(x_{i}^{*}+y_{i}^{*})t+\hat{\delta}t\right)\in I_{m}^{B}\right\}.
\end{align}
To prove Proposition \ref{proposition lower bound}, we show that there is a particle $X_k$ such that $X_k\in \bigcap_{m=1}^\ell \left(A_{m}^{k}\cap B_{m}^{k}\right)$ and $X_k(t)/t$ is larger than the r.h.s.\ of \eqref{lower bound almost sure convergence}. 
	\begin{proof}
		 First, assume $\sum_{i=1}^{\ell}(x_{i}^{*}+y_{i}^{*})<1$ and let w.l.o.g. $\epsilon/\sqrt{2}\in (0,1-\sum_{i=1}^{\ell}(x_{i}^{*}+y_{i}^{*}))$. We fix the additional time $\hat{\delta}>0$ such that $\hat{\delta}<\epsilon/\sqrt{2}$.
		 
	We introduce some events, see Figure \ref{fig:picture-lower-bound} for an illustration.
			\begin{figure}[H]
			\centering
			\includegraphics[width=1\linewidth]{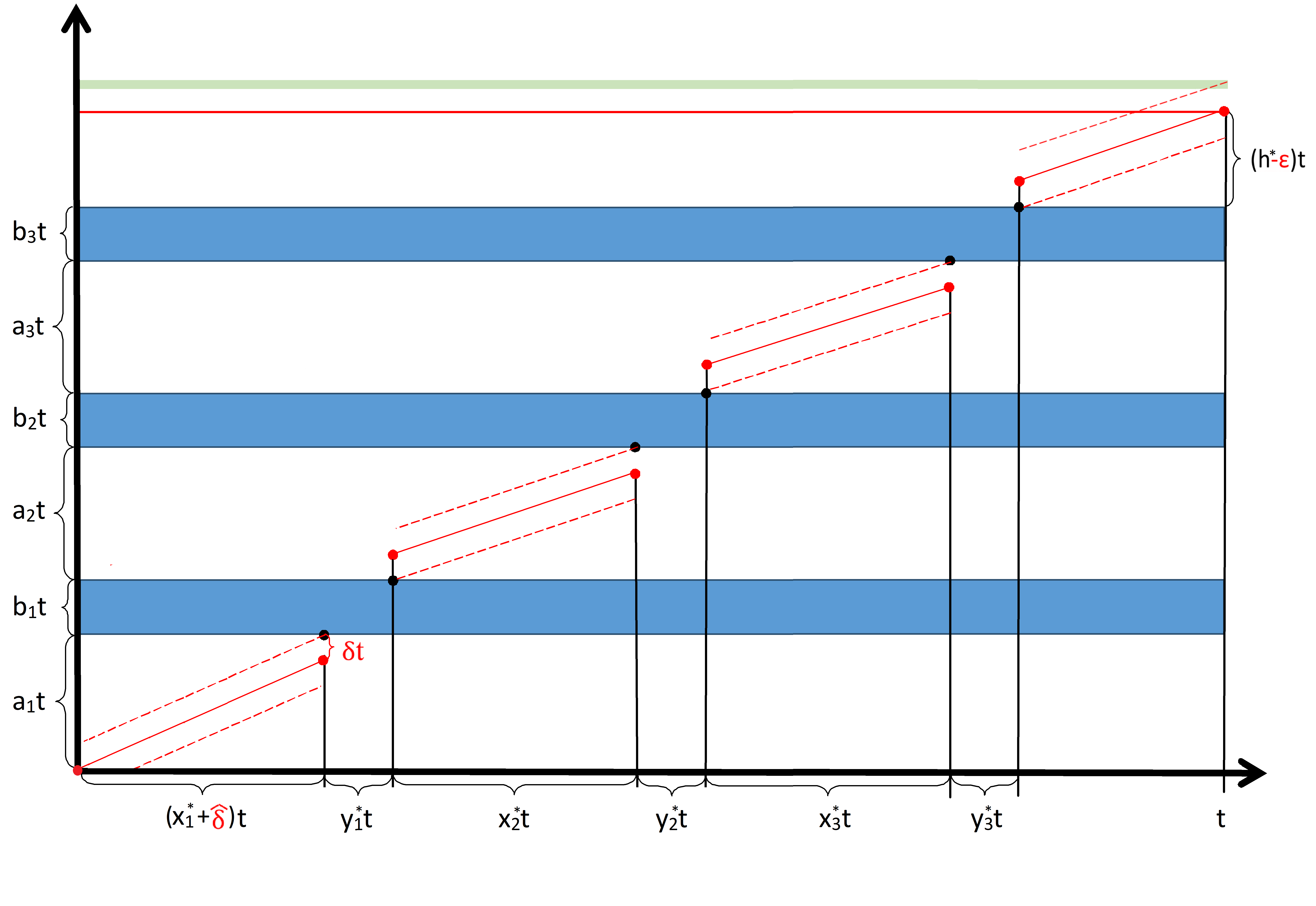}
			\caption[Proof idea lower bound]{The horizontal red line is at height $\sum_{i=1}^{\ell}(a_{i}+b_{i})t+(h^{*}-\epsilon)t$. To prove Proposition \ref{proposition lower bound}, we show that a certain number of particles follows the strategy marked with red dots.
			}
			\label{fig:picture-lower-bound}
		\end{figure}
	We define, for $m=1,\!...,\ell$ and some constant $C>0$,
		 \begin{align}
		 \mathcal{A}_{m}&=\left\{\exists t_{0}\forall t>t_{0}:\#\left\{k\leq n\left(\sum_{i=1}^{m-1}(x_{i}^{*}+y_{i}^{*})t+x_{m}^{*}t +\hat{\delta}t\right) : A_{m}^{k} \right\}\geq e^{\alpha_{m}t}\right\},
		 \\	\mathcal{B}_{m}&=\left\{\exists t_{0}\forall t>t_{0}:\#\left\{k\leq n\left(\sum_{i=1}^{m}(x_{i}^{*}+y_{i}^{*})t+\hat{\delta}t\right) : B_{m}^{k} \right\}\geq e^{\beta_{m}t}\right\},
				 \end{align}
		 where 
		\begin{align}
		\alpha_{m}&=\sum_{i=1}^{m}\left(x_{i}^{*}+\hat{\delta}\mathbbm{1}_{i=1}-\frac{(a_{i}-\delta)^{2}}{2(x_{i}^{*}+\hat{\delta}\mathbbm{1}_{i=1})}-\frac{(b_{i}+2\delta)^{2}}{2y_{i}^{*}}\right)+\frac{(b_{m}+2\delta)^{2}}{2y_{m}^{*}}
		-(2m-1)\gamma,
			\\\beta_{m}&=\alpha_{m}-\frac{(b_{m}+2\delta)^{2}}{2y_{m}^{*}}-\gamma.
		\end{align}
		Moreover, we define
		\begin{equation}
		 \mathcal{H}=\left\{\exists t_{0}\forall t>t_{0}\exists k\leq n(t) : X_{k}(t)\geq\left(\sum_{i=1}^{\ell}(a_{i}+b_{i})+h^{*}-\epsilon\right)t\right\}.
		\end{equation}
		 Since $(x_{1}^{*},y_{1}^{*},\!...,x_{\ell}^{*},y_{\ell}^{*})$ is in $D$, we have 
		\begin{align}
			\label{domain1 again}&\sum_{i=1}^{m}\left( x_{i}^{*}-\frac{a_{i}^{2}}{2x_{i}^{*}}-\frac{b_{i}^{2}}{2y_{i}^{*}}\right)\geq 0
			\mbox{ and }\sum_{i=1}^{m-1}\left( x_{i}^{*}-\frac{a_{i}^{2}}{2x_{i}^{*}}-\frac{b_{i}^{2}}{2y_{i}^{*}}\right)+x_{m}^{*}-\frac{a_{m}^{2}}{2x_{m}^{*}}\geq 0
		\end{align}
		for all $m=1,\!...,\ell$ . If we replace $x_{1}^{*}t$ by $(x_{1}^{*}+\hat{\delta})t$, the inequalities in \eqref{domain1 again} are strict. By continuity, we can choose the tube width parameter $\delta>0$ and the error parameter $\gamma>0$ so small, depending on $\hat{\delta}$, that $\alpha_{m}>0$ and $\beta_{m}>0$ for all $m=1,\!...,\ell$. 
	 The probability of $\mathcal{H}$ should go to one. By monotonicity and the Markov property, we have the lower bound
		\begin{align}\label{lower bound for P(mathcal H)}
			\mathbb{P}\left(\mathcal{H}\right)\geq& \mathbb{P}\left(\bigcap_{m=1}^{\ell}\left(\mathcal{A}_{m}\cap\mathcal{B}_{m}\right)\cap\mathcal{H}\right)
			\\\label{lower bound for P(mathcal H) part 2}=&\mathbb{P}\left(\mathcal{H}\middle|\mathcal{B}_{\ell}\right)
			\left(\prod_{m=2}^{\ell}\mathbb{P}\left(\mathcal{B}_{m}\middle|\mathcal{A}_{m}\right)
			\mathbb{P}\left(\mathcal{A}_{m}\middle|\mathcal{B}_{m-1}\right)\right)
			\mathbb{P}\left(\mathcal{B}_{1}\middle|\mathcal{A}_{1}\right)
			\mathbb{P}\left(\mathcal{A}_{1}\right).
		\end{align}
		By Lemmas \ref{lemma P(H | Bl)=1}, \ref{lemma P(Bm | Am)=1}, \ref{lemma P(Am |Bm-1)=1} and \ref{lemma for branching areas}, each factor of \eqref{lower bound for P(mathcal H) part 2} is equal to one. The claim \eqref{lower bound almost sure convergence} follows.
		
		If $\sum_{i=1}^{\ell}(x_{i}^{*}+y_{i}^{*})=1$, one could choose $\hat{\delta}>0$ so small that the inequalities in \eqref{domain1 again}  with $b_{\ell}$ replaced by $b_{\ell}-\epsilon$, $x_{1}^{*}$ replaced by $x_{1}^{*}+\hat{\delta}$ and $y_{\ell}^{*}$ replaced by $y_{\ell}^{*}-\hat{\delta}$ are strict. Then one could define $\alpha_{m}$, $\beta_{m}$, $A_{m}^{k}$, $B_{m}^{k}$, $\mathcal{A}_{m}$ and $\mathcal{B}_{m}$ as above but with $b_{\ell}+2\delta$ replaced by $b_{\ell}+\delta-\epsilon$ and $y_{m}^{*}$ replaced by $y_{m}^{*}-\hat{\delta}$. One could choose $\delta>0$ and $\gamma>0$ so small that $\alpha_{m}$ and $\beta_{m}$ are strictly positive. In \eqref{lower bound for P(mathcal H)}, one would have $\mathbb{P}\left(\mathcal{H}\middle|\mathcal{B}_{\ell}\right)=1$ by $\mathcal{B}_{\ell}\subset\mathcal{H}$. The remaining computations would work analogously.
	\end{proof}
	\begin{lem}\label{lemma P(H | Bl)=1}
		Under the assumption of Proposition \ref{proposition lower bound} and for $\sum_{i=1}^{\ell}(x_{i}^{*}+y_{i}^{*})<1$, we have $\mathbb{P}\left(\mathcal{H}\middle|\mathcal{B}_{\ell}\right)=1$.
	\end{lem}
	\begin{proof}
	By Lemma \ref{lemma for branching areas}, we can bound $\mathbb{P}\left(\mathcal{H}\middle|\mathcal{B}_{\ell}\right)$ by considering a BBM without obstacles. We have
		\begin{multline}\label{lower bound for P(mathcal H|B ell)}
			\mathbb{P}\left(\mathcal{H}\middle|\mathcal{B}_{\ell}\right)\geq \mathbb{P}\Bigg[\exists t_{0}\forall t>t_{0}\exists k\leq\hat{n}\left(\left(1-\sum_{i=1}^{\ell}(x_{i}^{*}+y_{i}^{*})-\hat{\delta}\right)t\right) : \\\widehat{X}_{k}\left(\left(1-\sum_{i=1}^{\ell}(x_{i}^{*}+y_{i}^{*})-\hat{\delta}\right)t\right)\geq\left(h^{*}-\epsilon-\delta\right)t\Bigg].
		\end{multline}
		Since we chose $\hat{\delta}<\epsilon/\sqrt{2}$, we have
		\begin{equation}
			\sqrt{2}\left(1-\sum_{i=1}^{\ell}(x_{i}^{*}+y_{i}^{*})-\hat{\delta}\right)>\sqrt{2}\left(1-\sum_{i=1}^{\ell}(x_{i}^{*}+y_{i}^{*})\right)-\epsilon-\delta.
		\end{equation}
		Hence, the r.h.s of \eqref{lower bound for P(mathcal H|B ell)} equals one by the tightness of the maximum of homogeneous BBM around $m(t)$, see \cite{B_C}.
	
	\end{proof}
	\begin{lem}\label{lemma P(Bm | Am)=1}
		Under the assumption of Proposition \ref{proposition lower bound}, we have $\mathbb{P}\left(\mathcal{B}_{m}\middle|\mathcal{A}_{m}\right)=1$ for $m=1,\!...,\ell$.
	\end{lem}
	\begin{proof}
		We show that
		\begin{equation}\label{propability that there are too few particles above obstacle_lower bound}
			1-\mathbb{P}\left(\#\left\{k\leq n\left(\sum_{i=1}^{m}(x_{i}^{*}+y_{i}^{*})t+\hat{\delta}t\right) : B_{m}^{k} \right\}\geq e^{\beta_{m}t}\middle|\mathcal{A}_{m}\right)
		\end{equation}
		is integrable with respect to $t$. This implies $\mathbb{P}\left(\mathcal{B}_{m}\middle|\mathcal{A}_{m}\right)=1$ for $m=1,\!...,\ell$ by the Borel-Cantelli Lemma and approximation arguments (see e.g. \cite{ABK_G}). 
		
		For $i=1,\!..., e^{\alpha_{m}t}$ with w.l.o.g. $e^{\alpha_{m}t}\in \mathbb{N}$, we define some independent Gaussian random variables $Y_{i}\sim \mathcal{N}\left(0,y_{m}^{*}t\right)$. By monotonicity, we can ignore possible branching and ask how many of the Gaussian random variables are in $\tilde{I}_{m}^{B}=[(b_{m}+2\delta)t,(b_{m}+2\delta)t+C]$. I.e. we bound \eqref{propability that there are too few particles above obstacle_lower bound} from above by 
		\begin{equation}\label{Gaussians for obstacles}
			1-\mathbb{P}\left(\sum_{i=1}^{e^{\alpha_{m}t}}\mathbbm{1}_{Y_{i}\in\tilde{I}_{m}^{B}}\geq e^{\beta_{m}+\gamma t-\gamma t}\right). 
		\end{equation}
		To apply the Paley-Zygmund inequality, we compute the expectation and the second moment. By Lemma \ref{lemma for obstacles}, we get
		\begin{equation}
			\mathbb{E}\left[\sum_{i=1}^{e^{\alpha_{m}t}}\mathbbm{1}_{Y_{i}\in\tilde{I}_{m}^{B}}\right]\approx \exp\left({\alpha_{m}t-\frac{(b_{m}+2\delta)^{2}t}{2y_{m}^{*}}}\right)=e^{\beta_{m}+\gamma t}.
		\end{equation}
		By Lemma \ref{lemma for obstacles} and the independence of $Y_{i}$, we have
		\begin{align}
			\mathbb{E}\left[\left(\sum_{i=1}^{e^{\alpha_{m}t}}\mathbbm{1}_{Y_{i}\in\tilde{I}_{m}^{B}}\right)^{2}\right]=\mathbb{E}\left[\sum_{i=1}^{e^{\alpha_{m}t}}\sum_{j=1}^{e^{\alpha_{m}t}}\mathbbm{1}_{Y_{i}\in\tilde{I}_{m}^{B}}\mathbbm{1}_{Y_{j}\in\tilde{I}_{m}^{B}}\right]
			\\\label{second moment obstacles}\approx e^{\beta_{m}+\gamma t}+e^{2(\beta_{m}+\gamma t)}
			-\exp\left({\alpha_{m}t-2\frac{(b_{m}+2\delta)^{2}t}{2y_{m}^{*}}}\right),
		\end{align}
		where the first summand of \eqref{second moment obstacles} considers the diagonal and the second and third summands consider all the other terms. Hence, by the Paley-Zygmund inequality, we can bound \eqref{Gaussians for obstacles} from above by
		\begin{equation}\label{paley zygmund obstacles}
			1-\mathbb{P}\left(\sum_{i=1}^{e^{\alpha_{m}t}}\mathbbm{1}_{Y_{i}\in\tilde{I}_{m}^{B}}\geq e^{\beta_{m}+\gamma t-\gamma t}\right)\leq1-\left(1-e^{-\gamma t}\right)^{2}\frac{e^{2(\beta_{m}+\gamma)t}}{e^{(\beta_{m}+\gamma) t}+e^{2(\beta_{m}+\gamma)t}}.
		\end{equation}
		Let $C_{6}>0$. The r.h.s of \eqref{paley zygmund obstacles} is smaller than $e^{-C_{6}t}$ if and only if
		\begin{equation}\label{paley zygmund obstacles bounded by exp(-C6t)}
			1-e^{-C_{6}t}+e^{-(\beta_{m}+\gamma)t}-e^{-(\beta_{m}+\gamma+C_{6})t}<1+e^{-2\gamma t}-2e^{-\gamma t}.
		\end{equation}
		By definition of $\beta_{m}>0$, we can choose $\gamma>0$ so small that still $\beta_{m}>0$ but also $\beta_{m}+\gamma>2\gamma$ for all $m=1,\!...,\ell-1$. Afterwards, we choose $C_{6}>0$ such that $C_{6}<\gamma$. Then \eqref{paley zygmund obstacles bounded by exp(-C6t)} is indeed true and we can bound \eqref{propability that there are too few particles above obstacle_lower bound} from above by $e^{-C_{6}t}$. 
	\end{proof}
	\begin{lem}\label{lemma P(Am |Bm-1)=1}
		Under the assumption of Proposition \ref{proposition lower bound}, we have $\mathbb{P}\left(\mathcal{A}_{m}\middle|\mathcal{B}_{m-1}\right)=1$ for $m=2,\!...,\ell$.
	\end{lem}
	\begin{proof}
		This proof is similar to the one of Lemma \ref{lemma P(Bm | Am)=1}. We want to show that 
		\begin{equation}\label{propability that there are too few particles above branching area_lower bound}
			1-\mathbb{P}\left(\#\left\{k\leq n\left(\sum_{i=1}^{m-1}(x_{i}^{*}+y_{i}^{*})t+x_{m}^{*}t +\hat{\delta}t\right) : A_{m}^{k} \right\}\geq e^{\alpha_{m}t}\middle|\mathcal{B}_{m-1}\right)
		\end{equation} 
		is integrable with respect to $t$. This implies $\mathbb{P}\left(\mathcal{A}_{m}\middle|\mathcal{B}_{m-1}\right)=1$ for $m=2,\!...,\ell$ by the Borel-Cantelli Lemma and approximation arguments (see e.g. \cite{ABK_G}). 
		
		We look at $e^{\beta_{m-1}t}$ independent BBMs that start in $0$ without obstacles. Denote by $\mathcal{Y}_{i}$ the event that the $i$-th BBM has at least $ \exp(x_{m}^{*}t-(a_{m}-\delta)^{2}t/(2x_{m}^{*})-\gamma t/3)$ particles in $[(a_{m}-\delta)t,a_{m}t]$ at time $x_{m}^{*}t$. By Lemma \ref{lemma for branching areas}, we can bound \eqref{propability that there are too few particles above branching area_lower bound} from above by
		\begin{equation}\label{3.68}
			1-\mathbb{P}\left(\sum_{i=1}^{ e^{\beta_{m-1}t}}\mathbbm{1}_{\mathcal{Y}_{i}}\geq e^{(\beta_{m-1}-2\gamma/3)t}\right).
		\end{equation}
		Since, by Lemma \ref{lemma for branching areas}, $\mathbb{P}(\mathcal{Y}_{i})\to 1$, as $t\to \infty$, it can be bounded from below by $\exp(-\gamma t/3)$ for large enough $t$. Analogously to \eqref{paley zygmund obstacles}, we can bound the probability in \eqref{3.68} from below by 
		\begin{align}
		& \mathbb{P}\left(\sum_{i=1}^{ e^{\beta_{m-1}t}}\mathbbm{1}_{\mathcal{Y}_{i}}\geq \mathbb{P}\left(\mathcal{Y}_{i}\right)e^{(\beta_{m-1}-\gamma/3)t}\right)
		\label{paley zygmund branching areas}&\geq
		\left(1-e^{-\gamma t/3}\right)^{2}\frac{e^{2\beta_{m-1}t}\mathbb{P}\left(\mathcal{Y}_{i}\right)^{2}}{e^{\beta_{m-1}t}\mathbb{P}\left(\mathcal{Y}_{i}\right)+e^{2\beta_{m-1}t}\mathbb{P}\left(\mathcal{Y}_{i}\right)^{2}},
		\end{align}
	using Paley-Zygmunds inequality and the independence of the BBMs.	Let $C_{7}>0$. The expression \eqref{paley zygmund branching areas} is smaller than $e^{-C_{7}t}$ if and only if
		\begin{equation}\label{paley zygmund branching areas bounded by exp(-C7t)}
			e^{-\beta_{m-1}t}-e^{-(C_{7}+\beta_{m-1})t}-e^{-C_{7}t}\mathbb{P}(\mathcal{Y}_{i})\leq e^{-\frac{2}{3}\gamma t}\mathbb{P}(\mathcal{Y}_{i})-2e^{-\frac{\gamma}{3}t}\mathbb{P}(\mathcal{Y}_{i}).
		\end{equation}
		By definition of $\beta_{m-1}>0$, we can choose $\gamma>0$ so small that still $\beta_{m-1}>0$ but also $2\gamma/3<\beta_{m-1}$ for all $m=2,\!...,\ell$. Afterwards, we choose $C_{7}>0$ such that $C_{7}<\gamma/3$. Then \eqref{paley zygmund branching areas bounded by exp(-C7t)} is indeed true and we can bound \eqref{propability that there are too few particles above branching area_lower bound} from above by $e^{-C_{7}t}$. 
	\end{proof}
\appendix
\section{Differentiability of $x_{m}^{c}$ }
	In this appendix, we show differentiability of $x_{m}^{c}$ with respect to $c_{m}$. Differentiability of $x_{m}^{c}$ with respect to $c_{m-1}$ can be proved analogously. For $(c_{1},\!...,c_{\ell})\in D^{c}$, we define $\mathcal{I}(c_{m})=\{C\in \mathbb{R}:(c_{1},\!...,c_{m-1},C,c_{m+1},\!...,c_{\ell})\in D^{c}\}$. First, we show that $x_{m}(C)$ is continuous in $C$ on $\mathcal{I}(c_{m})$ and a simple root of some polynomial. Then we use the implicit function theorem to show differentiability in the interior of $\mathcal{I}(c_{m})$. We start with some auxiliary results.
	\begin{cor}\label{corollary Icm interval}
		For all $(c_{1},\!...,c_{\ell})\in D^{c}$, $\mathcal{I}(c_{m})$ is an interval.
	\end{cor}
	\begin{proof}
		The claim follows from Lemma \ref{lemma Dc convex}.
	\end{proof}
	\begin{lem}\label{lemma Dmx(C) closed interval}
		For all $C\in \mathcal{I}(c_{m})$, the domain $D^{m}_{x}(c_{m-1},C)$ is a closed interval.
	\end{lem}
	\begin{proof}
		Connectivity of $D^{m}_{x}(c_{m-1},C)$ can be proved analogously to the proof of Lemma \ref{lemma Dc convex}. To show that $D^{m}_{x}(c_{m-1},C)$ is closed, let $x_{m}^{i}\in D^{m}_{x}(c_{m-1},C)$ be some convergent sequence. We have
		\begin{align}
			\label{xmci domain1}
			&c_{m-1}+x_{m}^{i}-\frac{a_{m}^{2}}{2x_{m}^{i}}-\frac{b_{m}^{2}}{2y_{m}^{i}}=C,
			\\\label{xmci domain3}&x_{m}^{i}+y_{m}^{i}\leq N,
			\\\label{xmci domain4}&x_{m}^{i}> 0 \text{ and } y_{m}^{i}> 0
		\end{align}
		with
		\begin{equation}\label{ym(i)}
			y_{m}^{i}=\frac{b_{m}^{2}}{2\left(c_{m-1}-C+x_{m}^{i}-\frac{a_{m}^{2}}{2x_{m}^{i}}\right)}.
		\end{equation} 
		Let $\hat{x}=\lim\limits_{i\to\infty}x_{m}^{i}$. By \eqref{xmci domain1}, it is not possible that $\hat{x}=0$ because $x_{m}^{i}\leq N$ is not able to compensate $a_{m}^{2}/(2x_{m}^{i})\to\infty$. Furthermore, the denominator of \eqref{ym(i)} can not converge to $0$ because $y_{m}^{i}\leq N$. Hence, $\hat{y}=\lim\limits_{i\to\infty}y_{m}^{i}$ is well defined. We have to show
		\begin{align}
			\label{xmclim domain1}
			&c_{m-1}+\hat{x}-\frac{a_{m}^{2}}{2\hat{x}}-\frac{b_{m}^{2}}{2\hat{y}}=C,
			\\\label{xmclim domain3}&\hat{x}+\hat{y}\leq N,
			\\\label{xmclim domain4}&\hat{x}> 0 \text{ and } \hat{y}> 0.
		\end{align}
		By \eqref{xmci domain1}, it is also not possible that $\hat{y}=0$ because $x_{m}^{i}\leq N$ is not able to compensate $b_{m}^{2}/(2y_{m}^{i})\to\infty$. Hence, \eqref{xmclim domain4} holds. The conditions \eqref{xmclim domain1} and \eqref{xmclim domain3} follow from \eqref{xmci domain1} and \eqref{xmci domain3} by continuity. Hence, $\hat{x}\in D^{m}_{x}(c_{m-1},C)$ and $D^{m}_{x}(c_{m-1},C)$ is closed.
	\end{proof}
	\begin{lem}\label{lemma Dmx(C) continuous in C}
		For all $\hat{C}\in \mathcal{I}(c_{m})$, all $\hat{x}\in D^{m}_{x}(c_{m-1},\hat{C})$ and all $\epsilon>0$ such that $\{x\in \mathbb{R}:|x-\hat{x}|<2\epsilon\}\subset D^{m}_{x}(c_{m-1},\hat{C})$, there exists $\delta>0$ such that for all $C\in \mathcal{I}(c_{m})$ with $|C-\hat{C}|<\delta$ we have $\{x\in \mathbb{R}:|x-\hat{x}|<\epsilon\}\subset D^{m}_{x}(c_{m-1},C)$.
	\end{lem} 
	\begin{proof}
		Since $D^{m}_{x}(c_{m-1},\hat{C})$ is closed by Lemma \ref{lemma Dmx(C) closed interval}, we have $\{x\in \mathbb{R}:|x-\hat{x}|\leq2\epsilon\}\subset D^{m}_{x}(c_{m-1},\hat{C})$. Hence, we have for all $\alpha\in[-2,2]$,
		\begin{align}
			\label{xmhat domain1}
			&c_{m-1}+\hat{x}+\epsilon\alpha-\frac{a_{m}^{2}}{2(\hat{x}+\epsilon\alpha)}-\frac{b_{m}^{2}}{2y(\alpha)}=\hat{C},
			\\\label{xmhat domain3}&\hat{x}+\epsilon\alpha+y(\alpha)\leq N,
			\\\label{xmhat domain4}&\hat{x}+\epsilon\alpha> 0 \text{ and } y(\alpha)> 0
		\end{align}
		with
		\begin{equation}\label{y(alpha)}
			y(\alpha)=\frac{b_{m}^{2}}{2\left(c_{m-1}-\hat{C}+\hat{x}+\epsilon\alpha-\frac{a_{m}^{2}}{2(\hat{x}+\epsilon\alpha)}\right)}.
		\end{equation} 
		We have to show that there is some $\delta>0$ such that for all $\beta \in(-1,1)$ and all $\alpha\in(-1,1)$,
		\begin{align}
			\label{xmhat delta domain1}
			&c_{m-1}+\hat{x}+\epsilon\alpha-\frac{a_{m}^{2}}{2(\hat{x}+\epsilon\alpha)}-\frac{b_{m}^{2}}{2y(\alpha,\beta)}=\hat{C}+\delta\beta,
			\\\label{xmhat delta domain3}&\hat{x}+\epsilon\alpha+y(\alpha,\beta)\leq N,
			\\\label{xmhat delta domain4}&\hat{x}+\epsilon\alpha> 0 \text{ and } y(\alpha,\beta)> 0
		\end{align}
		with
		\begin{equation}\label{y(alpha,beta)}
			y(\alpha,\beta)=\frac{b_{m}^{2}}{2\left(c_{m-1}-\hat{C}-\delta\beta+\hat{x}+\epsilon\alpha-\frac{a_{m}^{2}}{2(\hat{x}+\epsilon\alpha)}\right)}.
		\end{equation} 
		For $\alpha\in[-2,2]$, we define $F_{0}(\alpha)=2\left(c_{m-1}-\hat{C}+\hat{x}+\epsilon\alpha-\frac{a_{m}^{2}}{2(\hat{x}+\epsilon\alpha)}\right)$ and $\alpha^{0}=\text{argmin}\{F_{0}(\alpha):\alpha\in[-2,2]\}$. Note that $\alpha^{0}$ is well defined because $F_{0}$ is strictly concave on $[-2,2]$ by \eqref{xmhat domain4}. By \eqref{xmhat domain4} and \eqref{y(alpha)}, we have $F_{0}(\alpha^{0})>0$. Let $\hat{\delta}>0$ be so small that also $F_{0}(\alpha^{0})-2\hat{\delta}>0$. Then we have
		\begin{align}
			\frac{b_{m}^{2}}{2\left(c_{m-1}-\hat{C}-\hat{\delta}\beta+\hat{x}+\epsilon\alpha-\frac{a_{m}^{2}}{2(\hat{x}+\epsilon\alpha)}\right)}\geq\frac{b_{m}^{2}}{F_{0}(\alpha^{0})-2\hat{\delta}}>0
		\end{align} 
		for all $\beta \in[-1,1]$ and all $\alpha\in[-2,2]$.
		
		For $\delta\in(0,\hat{\delta})$ and $\alpha\in[-2,2]$, we define 
		\begin{equation}
			F_{1}(\delta,\alpha)=\hat{x}+\epsilon\alpha+\frac{b_{m}^{2}}{2\left(c_{m-1}-\hat{C}-\delta+\hat{x}+\epsilon\alpha-\frac{a_{m}^{2}}{2(\hat{x}+\epsilon\alpha)}\right)}.
		\end{equation}
		By \eqref{second derivative BEOx} in the proof of Proposition \ref{proposition optimal xmc existence and uniqueness}, $F_{1}(0,\alpha)$ is strictly convex in $\alpha$ on $[-2,2]$. Hence, $\alpha^{1}=\text{argmax}\{F_{1}(0,\alpha):\alpha\in[-1,1]\}$ is well defined and satisfies $F_{1}(0,\alpha^{1})<F_{1}(0,2)\leq N$ or $F_{1}(0,\alpha^{1})<F_{1}(0,-2)\leq N$.
		
		The derivative of $F_{1}$ with respect to $\delta$ is equal to
		\begin{equation}
			\partial_{\delta}F_{1}(\delta,\alpha)=\frac{b_{m}^{2}}{2\left(c_{m-1}-\hat{C}-\delta+\hat{x}+\epsilon\alpha-\frac{a_{m}^{2}}{2(\hat{x}+\epsilon\alpha)}\right)^{2}}.
		\end{equation}
		We bound $\partial_{\delta}F_{1}(\delta,\alpha)$ from above by $L=\partial_{\delta}F_{1}(\hat{\delta},\alpha^{0})$. Then we have
		\begin{align}
			|F_{1}(\delta,\alpha)-F_{1}(0,\alpha)|\leq L\delta.
		\end{align}
		Hence, for $\delta<\min\{\hat{\delta},(N-F_{1}(0,\alpha^{1}))/L\}$, we have 
		\begin{align}
			\hat{x}+\epsilon\alpha+y(\alpha,\beta)&=\hat{x}+\epsilon\alpha+\frac{b_{m}^{2}}{2\left(c_{m-1}-\hat{C}-\delta\beta+\hat{x}+\epsilon\alpha-\frac{a_{m}^{2}}{2(\hat{x}+\epsilon\alpha)}\right)}\leq F_{1}(\delta,\alpha),
		\end{align}
	which is bounded from above by $ F_{1}(\delta,\alpha)
	\leq F_{1}(0,\alpha)+L\delta
	\leq N$, 	for all $\beta \in(-1,1)$ and all $\alpha\in(-1,1)$.
	\end{proof}
	Looking at the first order condition \eqref{BEOx} with $c_{m}$ replaced by $C$ is equivalent to looking at roots of the polynomial
	\begin{equation}\label{polynomial with C instead of cm}
		x^{4}+2(c_{m-1}-C)x^{3}+\left((c_{m-1}-C)^{2}-a_{m}^{2}-\frac{b_{m}^{2}}{2}\right)x^{2}-a_{m}^{2}(c_{m-1}-C)x+\frac{a_{m}^{2}}{4}(a_{m}^{2}-b_{m}^{2}).
	\end{equation}
	By Proposition \ref{proposition optimal xmc existence and uniqueness}, we have the following corollary.
	\begin{cor}\label{corollary xm(C) largest real root, only root in Dmx(C), not on boundary}
		For all $C\in \mathcal{I}(c_{m})$, $x_{m}(C)$ is the largest real root of \eqref{polynomial with C instead of cm}, the only root of \eqref{polynomial with C instead of cm} in $D^{m}_{x}(c_{m-1},C)$ and not a boundary point of $D^{m}_{x}(c_{m-1},C)$.
	\end{cor}
	\begin{lem}\label{lemma four roots continuous in C}
		The four roots of \eqref{polynomial with C instead of cm} are continuous in $C$ on $\mathbb{R}$. 
	\end{lem}
	\begin{proof}
		The coefficients of \eqref{polynomial with C instead of cm} are continuous in $C$ on $\mathbb{R}$. This implies continuity of the roots by [II 5.2, \cite{K1976}].
	\end{proof}
	\begin{lem}\label{lemma nature of roots of quartic functions}
		Let $p(x)=a(4)x^{4}+a(3)x^{3}+a(2)x^{2}+a(1)x^{1}+a(0)$ be some polynomial with real coefficients $a(0),a(1),a(2),a(3),a(4)$ and $a(4)\neq 0$. Let $\Delta$ be the discriminant of $p$. Assume
		\begin{align}
			&8a(4)a(2)-3a(3)^{2}<0,
			\\&64a(4)^{3}a(0)-16a(4)^{2}a(2)^{2}+16a(4)a(3)^{2}a(2)-16a(4)^{2}a(3)a(1)-3a(3)^{4}<0.
		\end{align}
		Furthermore, assume that $\Delta=0$ implies $a(2)^{2}-3a(3)a(1)+12a(4)a(0)\neq 0$. Then $p$ has four simple real roots if $\Delta>0$, two simple real roots and two complex roots if $\Delta<0$, and two simple real roots and one real double root if $\Delta=0$.
	\end{lem}
	\begin{proof}
		The claim is a special case of \cite{doi:10.1080/00029890.1922.11986100}.
	\end{proof}
	Now, we use these auxiliary results to prove continuity and simplicity of $x_{m}(C)$.
	\begin{lem}\label{lemma xm(C) continuous and simple}
		For all $(c_{1},\!...,c_{\ell})\in D^{c}$, $x_{m}(C)$ is continuous in $C$ on $\mathcal{I}(c_{m})$ and a simple root of \eqref{polynomial with C instead of cm}.
	\end{lem}
	\begin{proof}
		We compute 
		\begin{multline}
			\Delta=\frac{1}{4}\Big(64a_{m}^{8}b_{m}^{4}-48a_{m}^{6}b_{m}^{6}+96a_{m}^{6}b_{m}^{4}(c_{m-1}-C)^{2}-15a_{m}^{4}b_{m}^{8}-48a_{m}^{4}b_{m}^{6}(c_{m-1}-C)^{2}
			\\+48a_{m}^{4}b_{m}^{4}(c_{m-1}-C)^{4}-a_{m}^{2}b_{m}^{10}
			+6a_{m}^{2}b_{m}^{8}(c_{m-1}-C)^{2}-12a_{m}^{2}b_{m}^{6}(c_{m-1}-C)^{4}+8a_{m}^{2}b_{m}^{4}(c_{m-1}-C)^{6}\Big),
		\end{multline}
		the discriminant of \eqref{polynomial with C instead of cm}. Note that $\Delta$ is continuous in $C$ on $\mathbb{R}$ and there are at most six $C$ such that $\Delta=0$. Between these roots, $\Delta$ does not change its sign. 
		By elementary algebraic manipulations
		and Lemma \ref{lemma nature of roots of quartic functions}, we have three cases:
		\begin{center}
			\begin{tabular}{|l|l|}
				\hline
				$\Delta>0$ & \eqref{polynomial with C instead of cm} has four simple real roots \\
				\hline
				$\Delta<0$ & \eqref{polynomial with C instead of cm} has two simple real roots and two complex roots \\
				\hline
				$\Delta=0$ & \eqref{polynomial with C instead of cm} has one double real root and two simple real roots \\
				\hline
			\end{tabular}
		\end{center}
		Within each connected component of $\{C\in\mathcal{I}(c_{m}):\Delta>0\}$, the four roots of \eqref{polynomial with C instead of cm} can not change their order by Lemma \ref{lemma four roots continuous in C}. Since $x_{m}(C)$ is the largest one by Corollary \ref{corollary xm(C) largest real root, only root in Dmx(C), not on boundary}, $x_{m}(C)$ is continuous in $C$ on $\{C\in\mathcal{I}(c_{m}):\Delta>0\}$ and a simple root.
		
		Assume there exists a sequence $C(i)$ such that $C(i)$ is in the same connected component of $\{C\in\mathcal{I}(c_{m}):\Delta>0\}$ for all $i$ and, as $i\to\infty$, we have $C(i)\to\hat{C}\in \mathcal{I}(c_{m})$ and $\Delta\searrow 0$. Then the roots of \eqref{polynomial with C instead of cm} stay real and simple, do not change their order and two of them converge to the same real number. By Corollary \ref{corollary xm(C) largest real root, only root in Dmx(C), not on boundary}, the root that represents $x_{m}(C(i))$ is in the interior of $D^{m}_{x}(c_{m-1},C(i))$ and the only root in $D^{m}_{x}(c_{m-1},C(i))$.
		By Lemmas \ref{lemma Dmx(C) closed interval}, \ref{lemma Dmx(C) continuous in C} and \ref{lemma four roots continuous in C}, the limit of the root that represents $x_{m}(C(i))$ for all $i$ is in $D^{m}_{x}(c_{m-1},\hat{C})$. By Corollary \ref{corollary xm(C) largest real root, only root in Dmx(C), not on boundary}, there is only one root in $D^{m}_{x}(c_{m-1},\hat{C})$ and this root is not a boundary point. Hence, the double root can not be in $D^{m}_{x}(c_{m-1},\hat{C})$ and $\lim\limits_{i\to\infty}x_{m}(C(i))=x_{m}(\hat{C})$ is a simple root.
		
		Within each connected component of $\{C\in\mathcal{I}(c_{m}):\Delta<0\}$, the two real roots of \eqref{polynomial with C instead of cm} can not change their order or become non real by Lemma \ref{lemma four roots continuous in C}. Since $x_{m}(C)$ is the larger one by Corollary \ref{corollary xm(C) largest real root, only root in Dmx(C), not on boundary}, $x_{m}(C)$ is continuous in $C$ on $\{C\in\mathcal{I}(c_{m}):\Delta<0\}$ and a simple root.
		
		Assume there exists a sequence $C(i)$ such that $C(i)$ is in the same connected component of $\{C\in\mathcal{I}(c_{m}):\Delta<0\}$ for all $i$ and, as $i\to\infty$, we have $C(i)\to\hat{C}\in \mathcal{I}(c_{m})$ and $\Delta\nearrow 0$. Then the two complex roots become a real double root and the other roots stay simple and real. As in the case of $\Delta\searrow 0$, the limit of the root that represents $x_{m}(C(i))$ for all $i$ is in the interior of $D^{m}_{x}(c_{m-1},\hat{C})$ and the only root in $D^{m}_{x}(c_{m-1},\hat{C})$. Hence, the double root can not be in $D^{m}_{x}(c_{m-1},\hat{C})$ and $\lim\limits_{i\to\infty}x_{m}(C(i))=x_{m}(\hat{C})$ is a simple root.
	\end{proof}
	\begin{lem}\label{lemma xmc differentiable with respect to cm}
		For all $(c_{1},\!...,c_{\ell})$ in the interior of $D^{c}$, $x_{m}^{c}$ is differentiable with respect to $c_{m}$.
	\end{lem}
	\begin{proof}
		Assume $(c_{1},\!...,c_{\ell})$ is in the interior of $D^{c}$. To apply the implicit function theorem, we introduce some notation.
		We define $V=\{x\in\mathbb{R}:|x-x_{m}(c_{m})|<\epsilon\}$ and $U=\{C\in\mathcal{I}(c_{m}):|C-c_{m}|<\delta\}$.
		Let $\epsilon>0$ be so small that the distance of $V$ to the boundary of $D^{m}_{x}(c_{m-1},c_{m})$ is at least $\epsilon$. This is possible because $x_{m}(c_{m})$ is not a boundary point by Corollary \ref{corollary xm(C) largest real root, only root in Dmx(C), not on boundary}. Then we choose $\delta>0$ so small that two things hold. First, $U$ is in the interior of $\mathcal{I}(c_{m})$. This is possible because $(c_{1},\!...,c_{\ell})$ is in the interior of $D^{c}$. Secondly, for all $C\in U$, $V$ is a subset of the interior of $D^{m}_{x}(c_{m-1},C)$. This is possible by Lemma \ref{lemma Dmx(C) continuous in C}. Furthermore, we define the function $F:U\times V\to\mathbb{R}$ such that $F(C,x)$ is equal to \eqref{polynomial with C instead of cm}. 
		
		We have $F(c_{m},x_{m}(c_{m}))=0$. Since $x_{m}(c_{m})$ is a simple root by Lemma \ref{lemma xm(C) continuous and simple}, we also have partial $\partial_{x}F(c_{m},x_{m}(c_{m}))\neq 0$ by [Proposition 1.49, \cite{DKRH2016}]. By the implicit function theorem, there exist open neighbourhoods $U_{0}\subset U$ of $c_{m}$ and $V_{0}\subset V$ of $x_{m}(c_{m})$ and a unique continuously differentiable function $\hat{F} :U_{0}\to V_{0}$ such that $\hat{F}(c_{m})=x_{m}(c_{m})$ and $F(c,x)=0$ if and only if $x=\hat{F}(C)$. 
		
		By Corollary \ref{corollary xm(C) largest real root, only root in Dmx(C), not on boundary}, $x_{m}(C)$ is the only root in $D^{m}_{x}(c_{m-1},C)$. Our construction ensures $V_{0}\subset V\subset D^{m}_{x}(c_{m-1},C)$ for all $C\in U$. By Lemma \ref{lemma xm(C) continuous and simple}, there exists an open neighbourhood $U_{1}\subset U_{0}$ of $c_{m}$ such that $x_{m}(C)\in V_{0}$ for all $C\in U_{1}$. This implies $\hat{F}(C)=x_{m}(C)$ for all $C\in U_{1}$. Hence, $x_{m}^{c}=x_{m}(c_{m})$ is differentiable with respect to $c_{m}$.
	\end{proof}
\bibliographystyle{abbrv}

\bibliography{obstacle.bib}
\end{document}